\documentclass[11pt,reqno,a4paper,psamsfonts]{amsart}
\usepackage[utf8]{inputenc}
\usepackage[T1]{fontenc}
\usepackage[english]{babel}
\usepackage{amssymb,amsmath,amscd,enumerate,amsthm}
\usepackage[psamsfonts]{amsfonts}
\usepackage{latexsym}
\usepackage{amsbsy}
\usepackage{ulem}
\usepackage[left=3.5cm,right=3.5cm,top=3.5cm,bottom=3.5cm]{geometry}
\usepackage[hidelinks]{hyperref}
\usepackage{todonotes}
\usepackage{frcursive}
\usepackage{pdfsync}
\usepackage{xcolor,hyperref}
\usepackage{relsize,exscale}
\hypersetup{backref,colorlinks=true,linkcolor=blue,citecolor=blue}
\usepackage{pdfsync}
\usepackage{enumerate}
\usepackage{verbatim}
\usepackage{color}
\usepackage[all,color ]{xy}
\usepackage{graphicx}
\newcommand{\point}{\mathbf{\cdot}}
\newcommand{\diam}{\diamond}
\newcommand{\C}{{\mathbb{C}}}

\newcommand{\F}{{\mathcal{F}}}
\newcommand{\R}{{\mathbb{R}}}%
\newcommand{\Z}{{\mathbb{Z}}}

\newcommand{\tz}{{\tilde{z}}}
\newcommand{\tu}{{\tilde{u}}}

\let\CAL=\mathcal%
\def\mathcal#1{{\CAL#1}}%
\newcommand{\N}{\mathbb{N}}
\newcommand{\Ve}{{\mathsf{Ve}}}
\newcommand{\Ed}{{\mathsf{Ed}}}

\let\ge\eg

\newcommand{\A}{{\mathsf{A}}}

\newcommand{\Cex}{$\mc C^{\mr{ex}}$}
\newcommand{\Ctop}{$\mc C^0$}

\newcommand{\G}{\mathcal{G}}

\renewcommand{\emph}[1]{{\textbf{#1}}}
\newcommand{\un}[1]{{\underline{#1}}}
\newcommand{\ov}[1]{{\overline{#1}}}

\newcommand{\mf}[1]{{\mathfrak{#1}}}
\newcommand{\mr}[1]{{{\mathrm{#1}}}}
\newcommand{\mc}[1]{{\mathcal{#1}}}
\newcommand{\mb}[1]{{\mathbb{#1}}}
\newcommand{\wt}[1]{{\widetilde{#1}}}

\newcommand{\mbf}[1]{\mathbf{#1}}
\let\msf\msl
\newcommand{\iso}{{\overset{\sim}{\longrightarrow}}}

\newcommand{\DD}[1]{\frac{\partial\phantom{ #1}}{\partial #1}}
\newtheorem{teo}{Theorem}[section]
\newtheorem{thmm}{Theorem}

\newtheorem{lema}[teo]{Lemma}

\newtheorem{prop}[teo]{Proposition}
\newtheorem{defin}[teo]{Definition}
\newtheorem{cor}[teo]{Corollary}

\newtheorem{obs2}[teo]{Remark}
\newtheorem{recap2}[teo]{Recapitulation}
\newtheorem{ex2}[teo]{Example}

\newenvironment{obs}{\begin{obs2}\rm}{\hfill\qed\end{obs2}}

\newenvironment{dem}{\begin{proof}[Proof]}{\end{proof}}
\newenvironment{dem2}[1]{\begin{proof}[Proof #1]}{\end{proof}}
\def\bibartp#1#2#3#4#5#6#7#8
{\bibitem{#1} {\sc #2}, {\it #3}, {#4}, {\bf t. #5}, pages { #6} to
{ #7}, {({#8})}}
\def\bibart#1#2#3#4#5#6
{\bibitem{#1} {\sc #2}, {\it #3}, {#4}, {\bf #5}, {({#6})}}
\def\bibliv#1#2#3#4#5
{\bibitem{#1} {\sc #2}, {\it #3}, {#4}, {({#5})}}
\def\bibaart#1#2#3#4
{\bibitem{#1} {\sc #2}, {\it #3}, {#4}}
\definecolor{vert}{rgb}{0,0.46,0}
\usepackage{amsfonts,mathrsfs}

\title[Complete families of foliation germs]{Topological Moduli Space for Germs of Holomorphic Foliations III: Complete Families}
\date{\today}
\author{David Mar\'{\i}n, Jean-Fran\c{c}ois Mattei and Eliane Salem}
\thanks{D. Mar\'{\i}n acknowledges financial support from the Spanish Ministry of Science, Innovation and Universities, through grant  PID2021-125625NB-I00 and  by the Agency for Management of University and Research Grants of Catalonia through the grants 2017SGR1725 and 2021SGR01015.  This work is also supported by the Spanish State Research Agency, through the Severo Ochoa and Mar\'ia de Maeztu Program for Centers and Units of Excellence in R\&D (CEX2020-001084-M). E. Salem acknowledges the support of CNRS through a delegation.}
\address{Departament de Matem\`{a}tiques\\ Universitat Aut\`{o}noma de Barcelona \\ E-08193 Cerdanyola del Vall\`es (Barcelona)\\ Spain\\ \newline \indent Centre de Recerca Matem\`atica, Campus de Bellaterra, E-08193 Cerdanyola del Vall\`es, Spain} \email{{David.Marin@uab.cat}}

\address{Institut de Math\'{e}matiques de Toulouse\\ Universit\'{e} Paul Sabatier\\ 118, Route de Narbonne\\ F-31062 Toulouse Cedex 9, France} \email{jean-francois.mattei@math.univ-toulouse.fr}
\address{Sorbonne Universit\'e, Universit\'e de  Paris, CNRS,  Institut de Math\'ematiques de Jussieu - Paris Rive Gauche, F-75005 Paris, France}\email{eliane.salem@imj-prg.fr}

\subjclass[2010]{Primary 37F75; Secondary 32M25, 32S50, 32S65, 34M} 

\keywords{Complex dynamical systems, complex ordinary differential equations, holomorphic vector fields, holomorphic foliations, singularities, holonomy,  moduli space, complete families, universal deformation.}

\usepackage{tikz}

\newcommand{\T}{\mathcal{T}}

\begin{document}
\maketitle
\begin{center}
 \today
\end{center}

\begin{abstract}
In this work we use our previous results on the topological classification of generic singular foliation germs on $(\C^{2},0)$ to construct complete families: 
after fixing the semi-local topological invariants  we prove the existence of a minimal family of foliation germs that contains all the topological classes and such that any equisingular global family with parameter space an arbitrary complex manifold factorizes through it.
\end{abstract}

\tableofcontents

\section{Introduction}\label{introduction}

This paper is the outcome of a series of three works  on the topological classification of germs of singular foliations in the complex plane. In  \cite{MMS}, after fixing the topological invariants already known \cite{MM3}, we have constructed a moduli space of topological classes. Then in \cite{MMSII}, we have studied small perturbations of a generic foliation by proving the existence of a ``topologically universal'' deformation germ and by representing the ``deformation functor''. In the present paper we rely on these two results and in particular on the algebraic structure of the moduli space highlighted in \cite{MMS} to obtain a ``complete family''. Here ``complete'' means  a family that contains all the topological classes and such that any equisingular family with parameter space an arbitrary complex manifold factorizes through it via a possibly multivalued map.\\

We shall only consider germs of  foliations on $(\C^2,0)$ that are 
``\emph{generalized curves}'' \cite{CLNS}, in the sense that on the \emph{exceptional divisor} $\mc E_\F:=E_\F^{-1}(0)$  of the reduction of singularities map $E_\F: M_\F\to \C^2$, there are no singularities of the  \emph{reduced foliation} $\F^\sharp:=E_{\F}^{-1}(\F)$ of  \emph{saddle-node} type\footnote{i.e. locally defined by a vector field germ whose linear part has exactly one non-zero eigenvalue.}, cf.~\cite{CCD,Frank}. 
The exceptional divisor $\mc E_\F$ may have \emph{dicritical} (i.e. non-invariant by $\F^\sharp$) irreducible components and $\F^\sharp$ may possess \emph{nodal}\footnote{i.e. locally defined by a vector field germ such that the ratio of the eigenvalues of its linear part is a strictly positive real number.}  singularities on $\mc E_\F$.

\medskip

We will say that two germs of foliations $\F$ and $\G$ have \emph{same SL-type} if there exists a homeomorphism $\varphi:\mc E_\F\iso\mc E_\G$ satisfying:\\[-1mm]
\begin{enumerate}[(SL1)]
\item\it $\varphi(\mr{Sing}(\F^\sharp))= \mr{Sing}(\G^\sharp)$ and  $\varphi(D)\cdot \varphi(D') = D\cdot D'$ for any irreducible components $D,D'$ of $\mc E_\F$,\rm\\[-2mm]
\item\it if $D \subset \mc E_\F$ is  a $\F^\sharp$-invariant component and $p\in D$ is a singular point of $\F^\sharp$, then the Camacho-Sad indices of $\F^\sharp$ at $p$ along $D$ and of $\G^\sharp$ at $\varphi(p)$ along $\varphi(D)$ are equal,\rm\\[-2mm]
\item\it if $\mc H^{\F}_D:\pi_1(D\setminus \mr{Sing}(\F^\sharp),\cdot)\to \mr{Diff}(\C,0)$ denotes the holonomy morphism of $\F^\sharp$ along an invariant component $D\subset \mc E_\F$ and $\varphi_\ast$ denotes the morphism induced by $\varphi$ at the fundamental groups level, then up to composition by an inner automorphism of $\mr{Diff}(\C,0)$,  $\mc H^{\F}_D\circ \varphi_\ast^{-1}$ is the  holonomy morphism of $\G^\sharp$ along $\varphi(D)$.
\end{enumerate}\medskip
When all germs of a family of foliations
have same SL-type, the family will be called \emph{equisingular}. This notion, which presupposes local equireduction of the considered family, is specifically defined in Section \ref{sec-equi-global}. For an equisingular family $\F_P$ with parameter space a complex manifold $P$, we  will denote by 
\[\F_P(t_0):=\F_{P}|_{t=t_0}\] the \emph{fiber} over a particular value $t_0\in P$ and  by 
$\F_{P,t_0}$ the germ of this family along the fiber  $\F_P(t_0)$.

\medskip

Through all the paper a \emph{\Ctop-conjugacy} between two foliations is a homeomorphism sending leaves into leaves and {\it preserving the orientation of the ambient spaces as well as the orientation of the leaves.}\\

The notion of \emph{tame foliation} will be defined in section~\ref{tame}. If we exclude some exceptional configurations of $\mc E_\F$, the set of differential forms defining tame foliations contains a Krull open dense set. 
Our main result asserts the existence in this context of an equisingular global family which is ``topologically complete''. We also describe the (minimal) redundancy of their topological classes.

\begin{thmm}\label{main}
Let $\F$ be a tame foliation germ. Then
there exist $\tau\in \mb N$, a quotient~$\mbf D$ of a finite product of totally disconnected subgroups of $\mb U(1):=\{|z|=1\}\subset \C$ and an equisingular global family $\un\F_{\mbf U}$ over $\mbf U=\C^\tau\times\mbf D$ such that
\begin{enumerate}
\item\label{modsurj} for any foliation $\G$ with same SL-type as $\F$, there exists $u_0\in\mbf U$ such that $\G$ is $\mc C^0$-conjugated to the fiber $\F_{\mbf U}(u_0)$,
\item\label{globalfactnonmark}  if $P$ is a connected and simply connected complex manifold, $t_0\in P$ and  $\un\G_{P}$ is an equisingular global family  whose fiber $\G_P(t_0)$ is $\mc C^0$-conjugated to a fiber $\F_{\mbf U}(u_0)$, then there exists a holomorphic map $\lambda:P\to {\mbf U}$ such that  $\lambda(t_0)=u_0$, and  for any $t\in P$ the germs of families $\un\G_{P,t}$ and $\lambda^\ast\un\F_{\mbf U,\lambda(t)}$ over the germ of manifold $(P, t)$  are $\mc C^0$-conjugated,
\item\label{redondance} there exist $p\in\N$, a holomorphic action $*$ of $\Z^p$ on $\mbf U$ and an action $\star$ of a subgroup $\mbf I$ of the mapping class group of $(\mc  E_\F,\mr{Sing}(\F),\cdot)$   on the quotient $\mbf U/\Z^p$, such that two fibers $\F_{\mbf U}(u_1)$ and $\F_{\mbf U}(u_2)$ are \Ctop-conjugated if and only if there exists $\dot\varphi\in\mbf I$ such that $\dot\varphi\star(\Z^p*u_1)=\Z^p*u_2$. In particular, for any $u_0\in\mbf U$ the set of $u\in\mbf U$ such that $\F_{\mbf U}(u)$ is $\mc C^0$-conjugated to $\F_{\mbf U}(u_0)$ is at most countable.
\end{enumerate}
\end{thmm}
The  \emph{mapping class group  of $(\mc  E_\F,\mr{Sing}(\F),\cdot)$} is the set of isotopy classes
$\dot\varphi$ of  homeomorphisms $\varphi:\mc E_\F\to\mc E_\F$  preserving the orientation, the singular set, $\varphi(\mr{Sing}(\F))=\mr{Sing}(\F)$, and the intersection product, $\varphi(D)\cdot\varphi(D')=D\cdot D'$ for any irreducible components $D,D'$ of $\mc E_\F$. This group is always countable. \\

 Tame foliations have the remarkable property that any two  topologically conjugated tame foliations are also conjugated by an \emph{excellent homeomorphism}, i.e. one that lifts through the reduction of singularities as a homeomorphism which is holomorphic at the non-nodal singular points, cf. Theorem~\ref{C0 = Cex}. This result extends to equisingular families of tame foliations:
 
\begin{thmm}\label{conjug}
Let $\un\F_Q$ and $\un\G_Q$ be two  equisingular global families of germs of foliations  over a complex manifold $Q$, whose fibers are tame.
Then the following properties are equivalent:
\begin{enumerate}
\item\label{weakconj}  for any $u\in Q$ the fibers  $\F_{Q}(u)$ and $\G_Q(u)$ are \Ctop-conjugated,
\item\label{exconj} \it for any $u\in Q$ the fibers  $\F_{Q}(u)$ and $\G_Q(u)$  are \Cex-conjugated,
\item\label{strongconj} the global families $\un\F_Q$ and $\un\G_Q$  are  locally \Cex-conjugated.
\end{enumerate} 
\end{thmm}
\noindent A \emph{\Cex-conjugacy} of families is a \Ctop-conjugacy of families that lifts through the local equireduction maps
as a homeomorphism which is holomorphic at the non-nodal singularities, cf. \S\ref{sec-equi-global}.\\

Theorems \ref{main} and \ref{conjug} will follow from analogous results (Theorems~\ref{C} and~\ref{teo91}) in the context of marked foliations for which we can use the  moduli space of \Cex-conjugacy classes of marked foliations 
constructed in \cite{MMS}.
A \emph{marking} of $\F$ by a marked divisor $(\mc E,\Sigma,\cdot)$ is a  homeomorphism $f:\mc E\to\mc E_\F$ such that $f(\Sigma)=\mr{Sing}(\F^\sharp)$ and $f(D)\cdot f(D')=D\cdot D'$, cf. \S\ref{marquages}.
When $\F$ and $\G$ are endowed with markings $f:\mc E\to\mc E_\F$ and $g:\mc E\to\mc E_\G$ by a common marked divisor $(\mc E, \Sigma, \cdot)$,  we will say that the marked foliations $(\F,f)$ and $(\G,g)$ have \emph{same marked SL-type} if conditions (SL1)-(SL3) are fulfilled and if moreover $g^{-1}\circ \varphi\circ f$ is isotopic to the identity of $\mc E$ relatively to $\Sigma$. The following analogue of Theorem~\ref{main} in the marked setting holds for the larger class of \emph{finite type foliations} introduced in \cite[\S6]{MMS} and  specified in \cite[\S5]{MMSII}; in this context we have uniqueness of the factorization $\lambda:P\to\mbf U$ of a marked family under a weaker topological condition on its parameter space $P$.

\begin{thmm}\label{C}
Let $\F^\diam=(\F,f)$ be a marked finite type foliation which is a generalized curve.  Then there exists a marked equisingular global family of foliations $\un\F_{\mbf U}^\diam$ over $\mbf U=\C^\tau\times\mbf D$ such that
\begin{enumerate}
\setcounter{enumi}{-1}
\item\label{C0}  $\mbf D$ is a quotient of a finite product of totally disconnected subgroups of $\mb U(1)$ and $\tau$ is the dimension of the cohomological space $H^1(\A_\F,\T_\F)$ (of a complex that we recall in (\ref{suiteH1}) of \S\ref{Suniv}) whose finiteness characterizes the finite type of $\F$, cf. \cite[Theorem~5.15]{MMSII},
\item\label{C1}
if $\G^\diam=(\G,g)$ is a marked foliation with same marked SL-type as $\F^\diam$, there exists $u_0\in\mbf U$ such that $\G^\diam$ is \Cex-conjugated to $\F_{\mbf U}^\diam(u_0)$,
\item\label{C2}
 if $P$ is a connected manifold satisfying  $H_1(P,\mb Z)=0$, $t_0\in P$ and  $\un\G_{P}^\diam$ is a  marked equisingular  global family  whose fiber $\G^\diam_P(t_0)$ is \Cex-conjugated to a fiber $\F_{\mbf U}^\diam(u_0)$  as marked foliations, then there exists a unique holomorphic map $\lambda:P\to {\mbf U}$  such that $\lambda(t_0)=u_0$ and  for any $t\in P$ the germs of marked   families $\un\G_{P,t}^\diam$ and $\lambda^\ast\un\F^{\diam}_{\mbf U,\lambda(t)}$, over the germ of manifold $(P, t)$,  are \Cex-conjugated.
\end{enumerate}
\end{thmm}

\noindent An analogue of Theorem~\ref{conjug} in the marked setting will be given in Theorem~\ref{teo91}.\\

We will also compare the conjugation notions of local families and deformations.
A \emph{deformation of $\F$} is the data of a family $\un\F_{P,t_0}$ over a germ of holomorphic manifold $P$ at a point $t_0$ and a biholomorphism that identifies $\F$ to the fiber $\F_{P}(t_0)$. A \emph{conjugacy of deformations} of $\F$ is a conjugacy of the associated families compatible with the corresponding biholomorphisms, cf. \S\ref{sec-equi-global}. In Theorem~\ref{equideffam} we show that this compatibility condition is automatically fulfilled in the context of marked germs of families.\\

The central point of the paper is Theorem~\ref{localuniv}. It states the \Cex-universality of the germ at any point of the parameter space of the global family $\un\F_{\mbf U}$ constructed in \cite{MMS} that contains all the topological types in a fixed SL-class. This property will be proven  by explicitly computing the Kodaira-Spencer map of this family at each point, that provides an infinitesimal characterization of \Cex-universality.\\

In Chapter~\ref{Sfact} we look at the problem of existence of factorizations of global families through $\un\F_{\mbf U}$. Since Theorem~\ref{localuniv} gives local factorizations, obtaining a global factorization is reduced to a gluing problem. The group structure of the moduli space obtained in \cite{MMS} allows to translate this one into a cohomological problem that can be solved under weak topological assumptions on the parameter space of the global family.\\

All the study in Chapters~\ref{SLU} and~\ref{Sfact}, leading to Theorem~\ref{C},  is made for marked families modulo \Cex-conjugacy and only under the  finite type assumption. But Theorems~\ref{main} and~\ref{conjug} concern non-marked global families and \Ctop-conjugacies.
To work with \Ctop-conjugacies we require additional (Krull generic) hypothesis defining tame foliations in \S\ref{tame}, which allow to prove Theorem~\ref{main}. The proof of Theorem~\ref{conjug} in Section~\ref{weak-strong} is based again on the group structure of the moduli space using the fact that the mapping class group of the exceptional divisor is countable.

\section{Locally universal family}\label{SLU}

\subsection{Equisingular global families and deformations}\label{sec-equi-global}
We call \emph{(global) family of (germs of) foliations   over  a complex manifold~$Q$}, not necessarily connected, the data 
\begin{equation*}
\un{\F}_Q:=(M, \pi, \theta, \F_{Q})
\end{equation*}
of a complex manifold $M$  with $\dim(M)=\dim(Q)+2$,  a holomorphic surjective submersion  $\pi:M\to Q$,  a holomorphic section $\theta : Q\to M$ of $\pi$, and
 a germ along  $\theta(Q)$ of a one dimensional  holomorphic foliation $\F_{Q}$ on $M$ whose leaves are contained in the fibers of~$\pi$.  We say that $(M,\pi)$ is a \emph{manifold over $Q$}. For each $u\in Q$ we consider, in  the fiber of $\pi$ over $u$,  the germ of foliation at $\theta(u)$ obtained by restricting~$\F_Q$:
\begin{equation*}\label{fiberfamfol}
M(u):=\pi^{-1}(u)\,,\quad
\F_Q(u)
:=\F_{Q}|_{(M(u),\theta(u))}
\,.
\end{equation*}
The family   is \emph{equireducible} if $\theta(Q)$ is the singular locus of $\F_Q$ and for any point $u_0\in Q$ there is  an open \emph{trivializing neighborhood} $W\ni u_0$ and  a map called \emph{(minimal) equireduction map over $W$}
\begin{equation*}\label{equiredlocal}
E_{\F_{W}} :M_{\F_{W}}\to M_{W}:=\pi^{-1}(W)
\end{equation*}
that is defined by a sequence of blow-ups with etale  centers over $W$, and whose 
restriction  to each fiber 
\[ 
M_{\F_W}(u):=\pi^{\sharp\,-1}(u)\,,\quad \pi^\sharp:=\pi\circ E_{\F_{W}}\,,\quad u\in W\,,
 \]
is exactly the minimal reduction map of $\F_{Q}(u)$, and moreover the singular locus of the reduced foliation $\F_W^\sharp$ in $M_{\F_W}$ is also etale over $W$. A more detailed definition of this notion is given in \cite[\S2.2]{MMSII} or in \cite[Chapter 10, step (vi)]{MMS}.  Up to shrinking the neighborhood $W$ of $u_0$,  the exceptional divisor   $\mc E_{\F_{W}}=E_{\F_{W}}^{-1}(\theta(W))$ and the singular locus of the reduced foliation $\F_{W}^\sharp$ in $M_{\F_{W}}$ are topologically trivial: there exists a \emph{trivializing homeomorphism over $W$}
\begin{equation}\label{trivdivex}
\Psi_{W}:\mc E_{\F_{W}}
\iso 
\mc E_{\F_{W}}(u_0)\times W\,,
\quad \mr{pr}_W\circ\Psi_W=\pi^\sharp_{|\mc E_{\F_W}}\,,
\quad
\mc E_{\F_{W}}(u):=\mc E_{\F_{W}}\cap\pi^{\sharp\,-1}(u)\,,
\end{equation}
that sends the singular locus of $\F_{W}^\sharp$ on the product  $\mr{Sing}(\F^\sharp_{W}(u_0))\times W$, with 
\begin{equation*}\label{fiberfolsharp}
\F_{W}^\sharp(u):=\F_{W}^\sharp|_{ \pi^{\sharp\,{-1}}(u)}\,.
\end{equation*}
Restricted to the fiber of $u\in W$, $\Psi_W$ provides a homeomorphism
that identifies the exceptional divisor  of the reduction of $\F_W(u)$, with that of $\F_W(u_0)$,
\[ \Psi_u : \mc E_{\F_W}(u)\iso \mc E_{\F_W}(u_0) \,. \]
Thus the \emph{holonomy} of the foliation $\F_W^\sharp(u)$ along an invariant  component $D_u=\Psi^{-1}_u(D_{u_0})$ may be considered as a morphism $\mc H_{D_u}$ from the fundamental group of $
D_{u_0}\setminus \mr{Sing}(\F_W(u_0))$ into the group $\mr{Diff}(\C,0)$ of germs of biholomorphisms of $(\C,0)$. 
\begin{defin}
We say that an equireducible family $\un\F_Q$ is \emph{equisingular at $u_0\in Q$} if there is a  trivializing neighborhood $W$ of $u_0$ such that  for any invariant irreducible component $D_{u_0}\subset \mc E_{\F_W}(u_0)$ and for any point $m_0\in\mr{Sing}(\F_W(u_0))\cap D_{u_0}$, we have:
\begin{enumerate}[(a)]
\item   there exist biholomorphisms $\ell_u\in\mr{Diff}(\C,0)$ depending holomorphically of $u\in W$ such that $\ell_u\circ \mc H_{D_u}(\cdot)\circ \ell_u^{-1}=\mc H_{D_{u_0}}(\cdot)$, 
\item  the  \emph{Camacho-Sad function } from $W$ to $\C$:
 \[ 
u\mapsto \mr{CS}( \F_W(u), D_u,m_u),\quad D_u:=\Psi_W^{-1}(D_{u_0}\times \{u\})
\,,\quad
m_u:=\Psi^{-1}_W(m_0,u)\,,
 \]
is constant.
\end{enumerate}
\end{defin}

A \emph{$\mc C^0$-conjugacy} between two global equireducible families $\un \F_Q=(M, \pi, \theta, \F_{Q})$ and $\un \F'_Q=(M', \pi', \theta', \F'_{Q})$ over the same parameter space $Q$, is a germ of  homeomorphism $\Phi:(M,\theta(Q))\iso (M',\theta'(Q))$ satisfying $\Phi(\F_Q)=\F'_Q$ and $\pi'\circ\Phi=\pi$. We also assume that $\Phi$ preserves the orientation of the ambient spaces and the orientation of the leaves.
We will say that $\Phi$ is  \emph{excellent} or \emph{of  class $\mc C^{\mr{ex}}$}, if  its lifting $\Phi^\sharp_W$ through any local equireduction maps $E_{\F_W}$ and $E_{\F'_W}$, $E_{\F'_W}\circ \Phi^\sharp_W=\Phi\circ E_{\F_W}$, extends to the exceptional divisors, providing a homeomorphism germ
\[ \Phi^\sharp_W:(M_{\F_W},\mc E_{\F_W})\iso (M_{\F'_W},\mc E_{\F'_W})
 \]
which is holomorphic at any singular point of the exceptional divisor $\mc E_{\F_W}$ and at any non nodal singular point  of the foliation $\F^\sharp_W$.
We will also say that \emph{$\Phi^\sharp_{W}$ is excellent}.
\bigskip

Let $\mu:P\to Q$ be a holomorphic map and let $\un\F_Q=(M,\pi,\theta,\F_Q)$ be a global family  of foliations over $Q$. We consider the fibered product $\mu^*M=M\times_Q P\subset M\times P$ 
with the projection $\mu^*\pi:M\times_Q P\to P$, 
\[\xymatrix{\mu^*M\ar[d]^{\mu^*\pi} \ar[r]^{\rho_\mu} & M\ar[d]^{\pi}\\
P\ar[r]^{\mu}& Q}\qquad\qquad \xymatrix{(\mu^*M)(t)\ar[d]\ar[r]^{\sim}_{\rho_\mu} & M(\mu(t))\ar[d]\\
t\ar@{|->}[r]& \mu(t)}\] 
and the section $\mu^*\theta=(\theta\circ\mu)\times\mr{id}_P:P\to M\times_Q P$.
Since the restrictions to each fiber of the canonical submersion $\rho_\mu$ are biholomorphisms, there is a unique one-dimensional foliation germ $\mu^*\F_Q$ on $\mu^*M$ along $(\mu^*\theta)(P)$, tangent to the fibers of $\mu^*\pi$, such that $\rho_\mu$ sends the leaves of $\mu^*\F_Q$ into the leaves of $\F_Q$. We will call $\mu^*\un\F_Q=(\mu^*M,\mu^*\pi,\mu^*\theta,\mu^*\F_Q)$ the \emph{pull-back} of the global family $\un\F_Q$ by the map $\mu:P\to Q$.
Equisingularity is a local property in the parameters, by 
\cite[Proposition~3.7]{MMSII} it is preserved by pull-back. Moreover, if two global equisingular families $\un\F_Q$ and $\un\G_Q$ are $\mc C^0$ (resp. \Cex) conjugated by a homeomorphism $\Phi$ then so are $\mu^*\un\F_Q$ and $\mu^*\un\G_Q$ by $\mu^*\Phi=\Phi\times\mr{id}_P$.\\

 Let  $u_0$ be a point of $Q$ and let $\F$ be a germ of foliation  at a point $m_0$ of a two dimensional complex manifold  $M_0$. An \emph{equisingular deformation of $\F$ over the germ of manifold $(Q,u_0)$} is the data $(\un\F_{Q,u_0}, \iota)$ of the germ at $\theta(u_0)$ of an equisingular family $\un\F_Q=(M, \pi, \theta, \F_{Q})$ together with the germ of an embedding $\iota :(M_0,m_0)\hookrightarrow (M,\theta(u_0))$ that sends  $\F$ to the restricted foliation germ $\F_{Q}(u_0)$ on the \emph{special fiber} $M(u_0)$.

\begin{defin}\label{conj-def}
 A  \emph{$\mc C^0$ (resp. $\mc C^{\mr{ex}}$) conjugacy between two equisingular deformations} $(\un\F_{Q,u_0},\iota)$ and $(\un\F'_{Q,u_0},\iota')$ is a $\mc C^0$ (resp. $\mc C^{\mr{ex}}$) conjugacy $\Phi$ between their associated families, $\Phi(\un\F_{Q,u_0})=\un\F'_{Q,u_0}$, such that $\Phi\circ\iota=\iota'$. We will denote by
 $\mr{Def}_\F^{Q^\point}$ the set of \Cex-conjugacy classes of equisingular deformations of $\F$ over the germ of manifold $Q^\point:=(Q,u_0)$.   
\end{defin}

If $\mu:(P,t_0)\to (Q,u_0)$ is a holomorphic map germ and
$(\un\F_{Q,u_0},\iota)$ is an equisingular deformation of $\F$ over $(Q,u_0)$, then
 $(\mu^*\un\F_{Q,u_0},\mu^*\iota)$ is an equisingular deformation of $\F$ over $(P,t_0)$ where $\mu^*\iota$ is defined by $\rho_\mu\circ \mu^*\iota=\iota$ (recall that the restriction of $\rho_\mu$ to the fiber over $t_0$ is a biholomorphism onto the fiber over~$u_0$).

\begin{defin}\label{defdefuniv} Let 
 $(\un\F_{Q^\point},\iota)$ be an equisingular deformation over a germ of manifold 
$Q^\point:=(Q,u_0)$,  
of a   foliation germ $\F$. We say that $(\un\F_{Q^\point},\iota)$ is a  \emph{$\mc C^{\mr{ex}}$-universal deformation of $\F$} if for any germ of manifold $P^\point=(P,t_0)$ and any equisingular deformation $(\un\G_{P^\point},\delta)$ of $\F$ over $P^\point$, there exists a unique germ of holomorphic map 
$\lambda : P^{\point}\to Q^{\point}$ such that the deformations $(\un\G_{P^\point},\delta)$ and
 $\lambda^\ast (\un\F_{Q^\point},\iota)$  of $\F$ are $\mc C^{\mr{ex}}$-conjugated.
\end{defin} 

%For each germ of manifold $Q^\point$ and each generalized curve foliation $\F$ let us denote by $\mr{Def}_\F^{Q^\point}$ the set of \Cex-conjugacy classes of equisingular deformations of $\F$ over $Q^\point$.   

\begin{obs}\label{rempreliminaires}
Notice that if $\mu :Q{'}^\point\to Q^\point$ is a germ of biholomorphism, the $\mc C^\mr{ex}$-universality of $(\un\F_{Q^\point},\iota)$ and of $\mu^\ast(\un\F_{Q^\point},\iota)$ are clearly equivalent. 
On the other hand, it directly results from the definition that the  $\mc C^{\mr{ex}}$-universality of $(\un\F_{Q^\point},\iota)$ only depends on its \Cex -class of conjugacy $[\un\F_{Q^\point},\iota]\in\mr{Def}_\F^{Q^\point}$. We will  then say that $[\un\F_{Q^\point},\iota]$ is \emph{$\mc C^{\mr{ex}}$-universal}. 
\end{obs}

\begin{teo}[{\cite[Theorem 3.11 and Corollary 6.8]{MMSII}}]\label{phi*}
Let $\F$ and $\G$ be foliations of finite type which are  generalized curves and let
$\phi$ be an excellent conjugacy between $\G$ and $\F=\phi(\G)$.
If  $(\un\F_{Q^\point},\iota)$ is an equisingular deformation of $\F$ over $Q^\point$, there is an equisingular deformation $(\un\G_{Q^\point},\delta)$ of $\G$ over $Q^\point$ and an excellent conjugacy of families $\Phi:\un\G_{Q^\point}\to\un\F_{Q^\point}$ such that $\Phi\circ \delta=\iota\circ\phi$. 
Moreover, the map 
\[\phi^*:\mr{Def}_\F^{Q^\point}\to\mr{Def}_\G^{Q^\point}\,,\quad [\un\F_{Q^\point},\iota]\mapsto[\un\G_{Q^\point},\delta]\] 
is well defined, bijective and sends any class of \Cex-universal deformation of $\F$ to a class of \Cex-universal deformation of $\G$.
\end{teo}

\subsection{Marked foliations and families}\label{marquages}

Now, we fix for all the sequel a \emph{marked divisor} 
$\mc E^\diam = (\mc E, \Sigma,\cdot)$
in the sense of \cite[\S2.1]{MMS}, i.e.  a connected compact complex curve with normal crossings $\mc E$, endowed with a finite subset $\Sigma$ of $\mc E$ and a symmetric map $\mr{Comp}(\mc E)^2\to \Z$, $(D,D')\mapsto D\cdot D'$, where $\mr{Comp}(\mc E)$ denotes the set of irreducible components of $\mc E $. The components of $\mc E$ without any point of $\Sigma$ are called \emph{dicritical components}, the others being called  \emph{invariant components}. 
The \emph{mapping class group $\mr{Mcg}(\mc E^\diam)$ of $\mc E^\diam$} is the group of isotopy classes $\dot\varphi$ relatively to $\Sigma$ of orientation preserving homeomorphism $\varphi:\mc E\to\mc E$ such that $\varphi(\Sigma)=\Sigma$ and $\varphi(D)\cdot\varphi(D')=D\cdot D'$.\\

A \emph{marked by $\mc E^\diam$ foliation}
 is a pair $\F^\diam=(\F,f)$ where 
\begin{itemize}
\item $\F$ is a germ (at $m_0$) of a holomorphic foliation on a  2-dimensional manifold  $(M_0, m_0)$,
\item  $f$ is an orientation preserving homeomorphism, called \emph{marking of $\F$}, from $\mc E$ to the exceptional divisor $\mc E_\F$ of the reduction  of $\F$ such that: $f(\Sigma)$ is the singular set $\mr{Sing}(\F^\sharp)$  of the reduced foliation $\F^\sharp$, and  $D\cdot D'$ is equal to the intersection number of $f(D)$ with $f(D')$ in $M_\F$ for any components $D$, $D'$ of $\mc E$. Moreover we  will also suppose that $f$ is holomorphic at each point of $\Sigma\cup\mr{Sing}(\mc E)$.
\end{itemize}
\noindent We assume that  there exists a foliation germ that can be marked by  $\mc E^\diam$.
%, consequently  the dual graph $\A_{\mc E}$  of $\mc E$ is necessarily  a tree.
\\

Two markings $f$ and $g$ of $\F$ by  $\mc E^\diam$ will be  called \emph{equivalent} if the homeomorphism $g^{-1}\circ f$  is isotopic to the identity map 
of $\mc E$
by 
an isotopy leaving fixed  $\Sigma$.
A \emph{$\mc C^{\mr{ex}}$-conjugacy between two marked by $\mc E^\diam$ foliations} $\F^\diam=(\F,f)$ and $\G^\diam=(\G,g)$ is a germ $\phi$  of $\mc C^{\mr{ex}}$-conjugacy 
between these foliation germs, $\phi(\F)=\G$, such that $g$ and 
$\phi^\sharp\circ f$ are equivalent markings by $\mc E^\diam$ of $\G^\diam$, $\phi^\sharp$ being the lifting of $\phi$ through the reduction maps.
We then write $\F^\diam\sim_{\mc C^{\mr{ex}}}\G^\diam$ and we will denote by $[\F^\diam]$ the \Cex-conjugacy class of $\F^\diam$.\\

A \emph{pre-marking by $\mc E^\diam$ of an equireducible global family $\un\F_{Q}$} is a collection
\[
(f_u)_{u\in Q},\qquad f_u:\mc E\iso \mc E_{\F_Q(u)}\,,
\]
of markings $f_u$  for each foliation $\F_{Q}(u)$. Two pre-markings  $(f_u)_{u\in Q}$ and $(g_u)_{u\in Q}$ of $\un{\F}_Q$ of the same global  family will be called \emph{equivalent} if for each $u\in Q$ the markings $f_u$ and $g_u$ of $\F_Q(u)$ are equivalent. 
A \emph{marking of an equireducible global family $\un\F_{Q}$} is a pre-marking that satisfies the following \emph{local coherence property}: 
at any point $u_0\in Q$ there is  an equireduction neighborhood $W$ of $u_0$ and a trivializing homeomorphism  $\Psi_W$ as in (\ref{trivdivex}) such that the pre-marking $(f_u)_{u\in W}$ of $\un\F_Q$ over $W$ is equivalent to the pre-marking $(\Psi_u^{-1}\circ f_{u_0}
)_{u\in W}$, where $\Psi_u :\mc E_{\F_Q(u)}\iso \mc E_{\F_Q(u_0)} $ is the restriction of $\Psi_W$ to the fiber over $u$.
A \emph{marked by $\mc E^\diam$ global family over a manifold $Q$} is the data 
\[
\un\F_{Q}^\diam=(\un\F_{Q},(f_u)_{u\in Q})
\]
of an equireducible global family over $Q$ and a marking by $\mc E^\diam$ of this family. The \emph{fiber at $u\in Q$} of $\un\F_{Q}^\diam$ is the marked by $\mc E^\diam$   foliation $\F_{Q}^\diam(u):=(\F_{Q}(u), f_u)$. 
\begin{obs}\label{prglobalmark}
One can check that the set over $Q$ of the  equivalence classes of markings by $\mc E^\diam$ of the foliations $\F_Q(u)$, $u\in Q$,   can be endowed with a topology such that 
 it becomes a covering over $Q$ (the local coherence property being equivalent to the existence of continuous local sections) and
 the markings of $\un\F_Q$ are continuous global sections. In particular:
\begin{enumerate}[(a)]
\item\label{remfamslmark}  when $Q$ is connected, two markings of $\un\F_Q$ are equivalent as soon as,  up to an isotopy leaving  $\Sigma$ invariant,  they coincide at  some point $u_0\in Q$,
\item\label{extmarking} when  $Q$ is connected and  simply connected, any marking $f_{u_0}$ of the foliation ${\F}_Q(u_0)$  for some $u_0\in Q$,  extends to a marking $(f_u)_{u\in Q}$ of $\un{\F}_Q$, that is unique up to equivalence. 
\end{enumerate}
\end{obs}

A   \Cex-conjugacy between two marked by $\mc E^\diam$ global families $(\un\F_Q,(f_u)_{u\in Q})$ and $(\un\G_Q,(g_u)_{u\in Q})$ is a \Cex-conjugacy of global families $\Phi(\un\F_Q)=\un\G_Q$  such that the restriction $\Phi_u$ of $\Phi$ to each fiber is a \Cex-conjugacy between the corresponding marked by $\mc E^\diam$ foliations, i.e. $g_u^{-1}\circ\Phi_u^\sharp\circ f_u$  is isotopic to the identity map of $\mc E$ relatively to $\Sigma$.\\

For  a marked by $\mc E^\diam$ foliation $\F^\diam=(\F,f)$ and  an invariant component $D$ of $\mc E$,
we will denote by $[\mc H_D^{\F^\diamond}]$ the class, up to composition by  inner automorphisms of $ \mr{Diff}(\C,0)$, of the group morphism

\begin{equation*}
\mc H_D^{\F^\diam}:\pi_1(D^\ast, o_D)\to \mr{Diff}(\C,0)\,, 
\quad o_D\in D^\ast:=D\setminus \mr{Sing}(\F^\sharp)\,,
\end{equation*}
where $\mc H_D^{\F^\diam}(\dot\gamma)$ is the holonomy of the foliation $\F^\sharp$ along the loop $f\circ \gamma$ in $f(D)$. We also call \emph{Camacho-Sad index of $\F^\diam$} at a point $m\in D\cap\Sigma$ and we write $\mr{CS}(\F^\diam,D,m)$ the Camacho-Sad index of $\F^\sharp$ along $f(D)$ at the point $f(m)$. 

\begin{defin}\label{deftypeSL}
Two marked by $\mc E^\diam$ foliations $\F^\diam$ and $\G^\diam$ are \emph{SL-equivalent}, and we denote $\F^\diam\sim_{\mr{SL}}\G^\diam$,  if for any invariant  component $D$ of $\mc E$  and for any point $m\in D\cap\Sigma$ we have:
\[[\mc H_D^{\F^\diam}]=[\mc H_D^{\G^\diam}]\quad\text{and}\quad\mr{CS}(\F^\diam, D, m)=\mr{CS}(\G^\diam, D, m)\,.\]
\end{defin}
Clearly $\sim_{\mr{SL}}$ is a weaker equivalence relation than $\sim_{\mc C^{\mr{ex}}}$  on the (non-empty) set $\mr{Fol}(\mc E^\diam)$ of marked by $\mc E^\diam$ foliation germs. 
We will denote by 
\begin{itemize}
\item $\mr{SL}(\F^\diam) :=\{\G^\diam\in\mr{Fol}(\mc E^\diam)\;;\; \G^\diam\sim_{\mr{SL}}\F^\diam\}$  the $\sim_{\mr{SL}}$-class of $\F^\diam$, 
\item $\mr{SL}_Q(\F^\diam)$
the collection of all  marked by $\mc E^\diam$ equisingular   global   families $\un\F^\diam_Q$ over $Q$ such that  any fiber $\un\F_Q^\diam(u)$, $u\in Q$ is SL-equivalent to $\F^\diam$.
\end{itemize}

 Notice that if $\varphi:\mc E\iso\mc E$ is an homeomorphism such that $\dot\varphi\in\mr{Mcg}(\mc E^\diam)$ then
 \begin{equation}
 \label{SL-phi}
 (\F,f)\sim_{\mr{SL}}(\G,g)\Longrightarrow(\F,f\circ\varphi^{-1})\sim_{\mr{SL}}(\G,g\circ\varphi^{-1}).
 \end{equation}

\begin{defin}\label{IF}
The mapping class group $\mr{Mcg}(\mc E^\diam)$ acts on the set $\mr{Fol}(\mc E^\diam)/\!\sim_{\mc C^{\mr{ex}}}$ of \Cex-conjugacy classes of marked by $\mc E^\diam$ foliations by
\begin{equation}\label{star}
\dot\varphi\star[\F,f]:=[\F,f\circ\varphi^{-1}].
\end{equation}
If $\F^\diam=(\F,f)\in\mr{Fol}(\mc E^\diam)$ the subgroup 
\begin{equation}\label{IF0}
\mbf I_{\F^\diam}:=\{\dot\varphi\in\mr{Mcg}(\mc E^\diam)\,;\, (\F,f\circ\varphi^{-1
})\sim_{\mr{SL}}(\F,f)\}
\end{equation}
 leaves invariant the set
\[ \mr{Mod}([\F^\diam])
:=\{[\G^\diam]\;;\;\G^\diam\in \mr{SL(\F^\diam)} \}\]
of all $\mc C^{\mr{ex}}$-conjugacy classes $[\G^\diam]$ of marked by $\mc E^\diam$ foliations  $\G^\diam\in \mr{SL(\F^\diam)}$, called \emph{topological moduli space of $[\F^\diam]$.} Then (\ref{star}) defines an action $\star$ of $\mbf I_{\F^\diam}$ on $\mr{Mod}([\F^\diam])$.
\end{defin}
It is easy to check that the right hand sides of (\ref{star}) and (\ref{IF0}) are well-defined, i.e. do not depend on the choice of the representatives of the classes $[\F,f]$ and $\dot\varphi$.

\begin{prop}\label{orbit-fiber}
Let $(\F_1,f_1)$ and $(\F_2,f_2)$ be two marked foliations in $\mr{SL}(\F^\diam)$. The non-marked foliation germs $\F_1$ and $\F_2$ are \Cex-conjugated if and only if there is $\dot\varphi\in\mbf I_{\F^\diam}$ such that $\dot\varphi\star[\F_1,f_1]=[\F_2,f_2]$. In other words, the orbits of $\mbf I_{\F^\diam}$ on $\mr{Mod}([\F^\diam])$ coincide with the fibers of the forgetful map $\mr{Mod}([\F^\diam])\to \mr{SL}(\F)/\!\sim_{\mathcal{C}^{\mr{ex}}}$, where
 $\mr{SL}(\F)$ denotes the set of foliation germs having the same SL-type than $\F$ as defined in the introduction by means of properties (SL1)-(SL3).

\end{prop}
\begin{proof}
Let $\phi:(M_{\F_1},\mc E_{\F_1})\to(M_{\F_2},\mc E_{\F_2})$ be an excellent homeomorphism conjugating the reduced foliations $\F_1^\sharp$ to $\F_2^\sharp$. If we set $\varphi:=f_2^{-1}\circ \phi\circ f_1:\mc E\to\mc E$ then the marked foliations $(\F_1,f_1\circ\varphi^{-1})$ and $ (\F_2,f_2)$ are \Cex-conjugated by $\phi$ and $\dot\varphi\in\mr{Mcg}(\mc E^\diam)$ satisfies the equality $\dot\varphi\star[\F_1,f_1]=[\F_2,f_2]$. It remains to prove that $\dot\varphi\in\mbf I_{\F^\diam}$. As
$(\F_1,f_1\circ\varphi^{-1})$ and $(\F_2,f_2)$ are \Cex-equivalent, they also are $\mr{SL}$-equivalent and we deduce that $(\F_1,f_1\circ\varphi^{-1})\sim_{\mr{SL}}(\F,f)$. On the other hand,
$(\F_1,f_1)\sim_{\mr{SL}}(\F,f)$ implies that $(\F_1,f_1\circ\varphi^{-1})\sim_{\mr{SL}}(\F,f\circ\varphi^{-1})$ thanks to (\ref{SL-phi}). Hence $(\F,f)\sim_{\mr{SL}}(\F,f\circ\varphi^{-1})$ and consequently $\dot\varphi\in\mbf I_{\F^\diam}$.
\end{proof}

\begin{obs}\label{extsl}
When $Q$ is connected, a marked equisingular global family $\un\F_Q^\diam$ belongs to $\mr{SL_Q(\F^\diam)}$ as soon as one of its fibers $\F_Q^\diam(u_0)$  belongs to $ \mr{SL}(\F^\diam)$. Indeed  the Camacho-Sad indices of $\F_{Q}^\diam(u)$ depend continuously on $u$ and they are determined up to $2\pi i\Z$ by the holonomy maps around the singular points. The constancy of $u\mapsto \left[\mc H_D^{\F^\diam_Q(u)}\right]$ follows from the equisingularity of $\un\F_Q$ and the coherence property of the marking.
\end{obs}

 \subsection{Local universality}\label{Suniv} Let us suppose now that   $\F^\diam=(\F,f)$  is a  marked by~$\mc E^\diam$ foliation with $\F$  a finite type generalized curve, on an ambient space $(M_0,m_0)$. Theorem D  in \cite[\textsection 2.6]{MMS} gives a description of the moduli space $\mr{Mod}([\F^\diam])$ as
a pointed set naturally endowed with an abelian group structure given by an exact sequence
\begin{equation}\label{excsequmod1}
0\to \Z^p\stackrel{\alpha}{\to}\C^\tau\stackrel{\Lambda}{\to}\mr{Mod}([\F^\diam])\stackrel{\beta}{\to} \mbf D\to 1\,, 
\end{equation}
where
\begin{enumerate}[(i)]
\item\label{formuletau}  $\tau$  is the dimension of the cohomological space $H^1(\A_\F,\mc T_{\F})$ whose definition is recalled below, see (\ref{suiteH1}).
\item\label{expliciterD}  
$\mbf D$ is a quotient of a finite product of totally disconnected subgroups of  $\mb U(1):=\{|z|=1\}\subset \C$, that according to~\cite{Perez-Marco}  can be uncountable when $\F^\sharp$ possesses a singularity which is non-linearizable and non-resonant, cf. \cite[Example 4, \S8]{MMS}.
\end{enumerate}
\noindent Notice that the subgroup $\alpha(\Z^p)\subset\C^\tau$ might be not discrete, cf. \cite[Example 2, \S8]{MMS}.\\

In order to give the definition of $H^1(\A_\F,\T_\F)$ we first need to introduce the \emph{dual graph} $\A_\F$ (resp. $\A_{\mc E}$) of the exceptional divisor $\mc E_\F$ (resp. $\mc E$). It is the tree with vertex set  $\Ve_{\A_{\F}}$ (resp. $\Ve_{\A_{\mc E}}$) formed by the irreducible components of $\mc E_\F$ (resp. $\mc E$), and edge set $\Ed_{\A_{\F}}$ (resp. $\Ed_{\A_{\mc E}}$) consisting in unordered pairs $\langle D,D'\rangle$ of distinct irreducible components of $\mc E_\F$ (resp. $\mc E$) with $D\cap D'\neq\emptyset$. We also consider the set of \emph{oriented edges} of $\A_{\F}$ (resp. $\A_{\mc E}$)
\[\mc I_{\A_\F}=\{(D,\ge)\in\Ve_{\A_\F}\times\Ed_{\A_\F}\,:\,D\in \ge\}\]
(resp. $\mc I_{\A_{\mc E}}=\{(D,\ge)\in\Ve_{\A_{\mc E}}\times\Ed_{\A_{\mc E}}\,:\,D\in \ge\}$).

Let us denote by $\un{\mc X}_{\F}\subset\un{\mc B}_{\F}$  the sheaves of tangent and \emph{basic}\footnote{i.e. whose flows leave the foliation $\F^\sharp$ invariant.} holomorphic vector fields of $\F^{\sharp}$ on $M_{\F}$, and consider the  vector spaces associated to the vertices $D$ and the edges $\langle D,D'\rangle$ of the dual graph $\A_{\F}$ of $\mc E_\F$:

\[\T_{\F}(D)=\lim\limits_{\stackrel{\longleftarrow}{\tiny D\subset U}}H^0(U,\un{\mc B}_{\F}/\un{\mc X}_{\F})\]
 if $D\subset \mc E_{\F}$ is not a dicritical component of $\F$ and zero otherwise,
\[\T_{\F}(\langle D,D'\rangle)=\lim\limits_{\stackrel{\longleftarrow}{\tiny D\cap D'\subset U}}H^0(U,\un{\mc B}_{\F}/\un{\mc X}_{\F})\] if $D\cap D'$ is not reduced to a nodal singularity of $\F^\sharp$ and zero otherwise. Here $U$  runs over all open sets of $M_{\F}$ containing $D$ or $D\cap D'$.
By definition, $H^1(\A_{\F},\T_{\F})$ is the
 $1$-cohomology vector space $\ker\partial^1/\mr{Im}\,\partial^0$ of the following complex
\begin{equation}\label{suiteH1}\bigoplus_{D\in\Ve_{\A_{\F}}}\T_{\F}(D)\stackrel{\partial^0}{\longrightarrow}\bigoplus_{(D,\langle D,D'\rangle)\in\mc I_{\A_{\F}}}\T_{\F}(\langle D, D'\rangle)\stackrel{\partial^1}{\longrightarrow}\bigoplus_{\langle D,D'\rangle\in\Ed_{\A_{\F}}}\T_{\F}(\langle D,D'\rangle)),
\end{equation}
where 
\[ \partial^0((X_D)_D)=(X_{D'}-X_D)_{(D,\langle D,D'\rangle)}\,, \] 
\[\partial^1((X_{D,\langle D,D'\rangle})_{(D,\langle D,D'\rangle)})=(X_{D,\langle D,D'\rangle}+X_{D',\langle D,D'\rangle})_{\langle D,D'\rangle}.\]
If $\F$ is of finite type then $H^1(\A_{\F},\T_{\F})$ is of finite dimension  by \cite[Theorem~5.15]{MMSII}.

\begin{defin}\label{moduli-map}
We call \emph{moduli map} of a marked global  family $\un{\F}^\diam_Q\in \mr{SL}_Q(\F^\diam)$ the map 
\begin{equation*}\label{modmap}
\mr{mod}_{\un\F^\diam_Q} : Q \to  \mr{Mod}([\F^\diam])\,,
\quad
u\mapsto [ {\F}^\diam_{Q}(u) ]\,.
\end{equation*}
\end{defin}

\noindent We proved in \cite{MMS} that  for any map $\zeta:\mbf D\to \mr{Mod}([\F^\diam])$ such that $\beta\circ \zeta=\mr{id}_{\mbf D}$, there exists a marked equisingular  global family 
\begin{equation*}\label{FUIMRN}
\un\F^{\diam}_{\mbf U}=(\un\F_{\mbf U}, (f_{z,d})_{z,d})\in \mr{SL}_{\mbf U}(\F^\diam)\,,\quad
\un\F_{\mbf U}=
\left(M_0\times {\mbf U},\pi,\theta,\F_{\mbf U}\right)
\,,
\end{equation*}
\[ \mbf U:= \C^\tau\times\mbf D\,,\quad
\pi:M_0\times {\mbf U}\to
{\mbf U}\,, 
\quad 
\pi(m,z,d):=(z,d)\,,
\quad
\theta(z,d):=(m_0,z,d)\,,
\]
where $\mbf D$ is  endowed with the discrete topology, such that if we denote by the dot $\cdot$ the group operation in $\mr{Mod}([\F^\diam])$, we have:
\begin{equation}\label{propLambda} \mr{mod}_{\un\F^{\diam}_{\mbf U}}(z,d)=\Lambda(z)\cdot \zeta(d)\,.
\end{equation}
The goal of this section is to prove that this global family satisfies a local universal property:

\begin{teo}\label{localuniv}
Let $\F^\diam=(\F,f)$ be a marked by $\mc E^\diam$ foliation of finite type which is a generalized curve. Let 
 $\zeta$ be a section of the map $\beta:\mr{Mod}([\F^\diam])\to\mbf D$ in the exact sequence (\ref{excsequmod1}). Then there exists a  marked  global family 
 $\un\F^{\diam}_{\mbf U}=(\un\F_{\mbf U}, (f_{z,d})_{z,d})\in \mr{SL}_{\mbf U}(\F^\diam)$, $\mbf U:=\C^\tau\times \mbf D$,   
 such that
\begin{enumerate}
\item\label{thmmod} the moduli map $\mr{mod}_{\un\F_{\mbf U}^\diam}$ is surjective and relation (\ref{propLambda}) is satisfied,
\item\label{thmlocuniv}  for any point $\tu
\in {\mbf U}$, 
the deformation $(\un\F_{\mbf U,\tu},\iota_{\tu})$ of the foliation $\un\F_{\mbf U}(\tu)$ 
over the germ of manifold $({\mbf U},\tu)$, 
given by the germ 
of $\un\F_{\mbf U}$ at $\tu$ and the canonical  embedding $\iota_\tu:M_0\times\{\tu\}\hookrightarrow M_0\times\mbf U$,
%$\iota_\tu(m):=(m,\tu)$,
is
$\mc C^{\mr{ex}}$-universal.
\end{enumerate}
\end{teo}

\begin{dem2}{of Theorem~\ref{localuniv}}
We will see that the marked  global family  $\un\F_{\mbf U}^{\diam}$ introduced in \cite{MMS}, which fulfills property 
 (\ref{thmmod}), also  satisfies the assertion~(\ref{thmlocuniv}). For this we will use the criterion of universality given in  \cite[Theorem~6.7]{MMSII} by showing that for any $\tu\in\mbf U$ the Kodaira-Spencer map  \cite[\S6.2]{MMSII} of the germ  $\un\F_{\mbf U,\tu}$
  is an isomorphism.
In the first step we recall the process of construction of $\un\F_{\mbf U}^{\diam}$ made in \cite[\textsection 10, Step (vii)]{MMS}. 
In the second step we will determine a ``good trivializing system'' 
 for $\un\F_{\mbf U,\tu}$ which will be used in the last step of the proof to compute the  Kodaira-Spencer map of $\un\F_{\mbf U,\tu}$.\\

\textit{-Step 1. } Let us fix $d\in \mbf D$ and a marked foliation $\G^\diam=(\G,g)$ belonging to $\zeta(d)$.  First, let us recall that there are germs of \Cex-homeomorphisms compatible with the markings
\begin{equation}\label{conjDGF}
\psi_D:(M_\G, g(D))\iso(M_\F, f(D))\,,\quad D\in\Ve_{\A_{\mc E}}\,,
\end{equation}
that conjugate $\G^\sharp$ to $\F^\sharp$. The biholomorphism germs
\begin{equation}\label{cocycleG}
\phi_{D,\langle D,D'\rangle}:= \psi_{D'}\circ\psi_D^{-1} : (M_\F,s_{\langle D,D'\rangle})\iso (M_\F,s_{\langle D,D'\rangle})\,,
\end{equation}
with
\[
\{s_{\langle D,D'\rangle}\}:=f(D\cap D')\,,\quad \langle D,D'\rangle\in \Ed_{\A_{\mc E}}\,,
\] 
leave $\F^\sharp$ invariant. 
Thanks to the following lemma
we may also require that $\phi_{D,\langle D,D'\rangle}$ is the identity map when  
$s_{\langle D,D'\rangle}$ is a nodal singular point of $\F^\sharp$ or a singular point of the divisor belonging to a dicritical component. 

\begin{lema}
Let $\F$ be the foliation on $\C^2$ defined by  $z_1dz_2-\alpha z_2 dz_1$ with $\alpha\in\R^+\setminus\mb Q$ (resp. by $dz_1$).  Denote $\pi_j(z_1,z_2)=z_j$, $j=1,2$, and for $c>0$, $K_c=\{|z_2|<c|z_1|^\alpha\}$ (resp. $K_c=\{|z_1|<c\}$).
Let $g^j$, $j=1,2$, be holomorphic automorphisms of $\F$ defined on the polydisk $P=\{|z_1|<1,|z_2|<1\}$. Assume that for $j=1,2$ we have $\pi_j\circ g^j=\pi_j$ (resp. $\pi_1\circ g^j=k^j\circ \pi_1$ for some holomorphic maps $k^j:\mb D=\{|z|<1\}\to\C$).  Then, for every $0<c_1<c_2<1$ there exists a $\mathcal C^0$-automorphism $g$ of $\F$ defined on $P$
 such that
$g=g^1$ on $K_{c_1}$ and $g=g^2$ on $P\setminus K_{c_2}$.
\end{lema}
This result follows from the arguments given in  \cite[\S8.5]{MM3} for the nodal case and \cite[p. 147]{MM4} for the dicritical case. Indeed, it can also be deduced from the proofs of Remarks~3.9 and~3.10 in the arXiv version of \cite{MMSII} which remain valid for non  parametric versions.\\

Using the marking $f:\mc E\to\mc E_\F$  we consider the set
\begin{equation}\label{SF}
S_{\F^\diam}:=\{\star\in\Ve_{\A_\mc E}\cup\Ed_{\A_{\mc E}}:\dim_\C \T_{\F}(f(\star))=1\}.
\end{equation}
We then choose a subset 
\[\mc A''\subset\mc A':= \{\langle D,D'\rangle\in S_{\F^\diam}\,|\, D \text{ or } D'\notin S_{\F^\diam}\}\]
obtained by removing from $\mc A'$ an element in each connected component of $S_{\F^\diam}$ not reduced to a single edge, cf \cite[\S2.6]{MMSII}. Finally we choose for each $\ge\in\mc A''$ one vertex $D\in\ge$. This gives us an orientation for each edge $\ge\in\mc A''$. We consider
\[\mc A:=\{(D,\ge): \ge\in\mc A''\}\subset \mc I_{\A_{\mc E}}\,.\]
Such a set will be called \emph{set of active  oriented  edges for $\F^\diam$}. We denote by $\tau$ the cardinality of $\mc A$ and we fix a bijection 
\begin{equation}\label{kappa}
\kappa:\mc A\to\{1,\ldots,\tau\}.
\end{equation}
For each $(D,\ge)\in\mc A$ we also fix a germ  $X_{\ge}$  at $s_\ge$ of basic and not tangent holomorphic vector field to $\F^\sharp$ on $M_\F$.

Taking into account that $\mbf D$ is discrete, to recall the construction of  the marked global  family $\un\F_{\mbf U}^{\diam}$, it suffices to fix $d\in\mbf D$ and to
describe its restriction
 $\un\F^\diam_{\C^\tau}$  to the connected component $\C^\tau\times\{d\}$ of $\mbf U$.
Denoting by $U_D$, $D\in\Ve_{\A_{\mc E}}$, a neighborhood of $f(D)$ in $M_\F$, the ambient space of $\un\F_{\C^\tau}$ is obtained by   gluing  the  neighborhoods $U_D\times \C^\tau$ of $f(D)\times \C^\tau$ in $M_\F\times \C^\tau$,  using an  appropriate family of biholomorphisms
\[
\mf u:=(\Phi_{D,\ge} )_{D\in \ge\in\Ed_{\A}}\,,
\]
\[\Phi_{D,\ge} : 
\left(U_D \times \C^\tau,\{s_{\ge}\}\times \C^\tau\right)\iso 
\left(U_{D'}\times \C^\tau ,\{s_{\ge}\}\times \C^\tau\right)\,,\quad\text{with } \msf e=\langle D,D'\rangle\,.
\]
More precisely, writing  $(m,t)\mapsto \exp(Z)[t](m)$  the flow at the time $t$ of a vector field $Z$, we set:  
\begin{itemize}
\item $\Phi_{D,\ge}(m,z)=(\phi_{D,\ge}(m),z)$, $z=(z_1,\ldots,z_\tau)$,  if $(D,\ge),(D',\ge) \notin \mc A$,
\item $\Phi_{D, \ge}(m,z)=(\phi_{D,\ge}\,\circ\, \mr{exp}(X_{\ge})[z_{\kappa(\msf e)}](m),z)$, if $(D,\ge)\in \mc A$, 
\item $\Phi_{D, \ge}=\Phi^{-1}_{D', \ge}$ if $(D',\ge) \in \mc A$,
\end{itemize}
where $\phi_{D,\ge}$ are the biholomorphism germs in (\ref{cocycleG}).
We consider the following germ of manifold 
\[\left
(M_{\mf u}, \mc E_{\mf u}\right) :=
\Big( \bigcup_{D\in\Ve_{\A_{\mc E}}} U_D  \times \C^\tau \times \{D\}\Big/ \sim_{\mf u}\;, \; \bigcup_{D\in\Ve_{\A_{\mc E}}}f(D)\times\C^\tau\times\{D\} \Big/ \sim_{\mf u}\Big)\,,
\]  
the equivalence relation $\sim_{\mf u}$ being defined by: 
\[U_D \times \C^\tau\times\{D\} \;\ni\;(m,z, D)
\sim_{\mf u}\; 
(\Phi_{D,\langle D,D'\rangle}(m,z), D')\,\in\;U_{D'} \times \C^\tau\times \{D'\}
\]
when $(m,z)$ belongs to the domain of $\Phi_{D,\langle D,D'\rangle}$.
% and with the single relation $(m,z,D)\sim_{\mf u} (m,z,D)$ otherwise.
As the biholomorphisms  $\Phi_{D,\ge}$ leave invariant the projections $U_D\times \C^\tau\to \C^\tau$ and  the constant family $\F^{\sharp\,\mr{ct}}_{\C^\tau}$, the gluing process provides a holomorphic  submersion $\pi_{\mf u}: M_{\mf u}\to \C^\tau$ and a  foliation tangent to the fibers of  $\pi_{\mf u}$, which we  denote by $\F_{\mf u}$. 

The ambient space of   $\un\F_{\C^\tau}$ is the manifold  over $\C^\tau$, obtained by   contracting  $\mc E_{\mf u}$ to a $\tau$-dimensional manifold $S_{\mf u}$:
\[
E_{\mf u} : (M_{\mf u},\mc E_{\mf u})\longrightarrow  (M^\flat_{\mf u}, S_{\mf u}) \,,
\quad
\pi_{\mf u}^\flat : (M^\flat _{\mf u},S_{\mf u})\longrightarrow \C^\tau\,,\quad
\pi_{\mf u}^\flat\circ E_{\mf u} := \pi_{\mf u}
\,.
\]
 Restricted to $ S_{\mf u}$  the submersion $\pi_{\mf u}^\flat$ is a biholomorphism, we will denote by $\theta_{\mf u}$ its inverse. Finally the foliation   $\F_{\C^\tau}$ constructed in \cite[\S 10, Step (vii)]{MMS} is the direct image $E_{\mf u}(\F_{\mf u})$ and we have an equisingular global family over $\C^\tau$ 
\begin{equation*}
\un\F_{\C^\tau} := \left(M^\flat_{\mf u}, \pi_{\mf u}^\flat,\theta_{\mf u}, \F_{\C^\tau}\right)\,.
\end{equation*}
By a classical property of blow-ups, there is a germ $F$ of  biholomorphism 
 that conjugates  the reduced foliation $\F_{\C^\tau}^{\sharp}$ to $\F_{\mf u}$: 
 \begin{equation}\label{conjprunivblup}
 F:(M_{\un\F_{\C^\tau}}, \mc E_{\F_{\C^\tau}} )\iso (M_{\mf u}, \mc E_{\mf u})\,,
\quad
F(\F_{\C^\tau}^{\sharp})=\F_{\mf u}\,,
 \end{equation}
In order to endow this family with a marking by $\mc E^\diam$ we highlight that we have:   
\[\Phi_{D,\ge}(m,0)=(\phi_{D,\ge}(m),0)\,,\quad D\in \ge\in\Ed_{\A_{\mc E}}\,.
\]
Therefore, according to relations (\ref{cocycleG})  the maps $\psi_D$ introduced in (\ref{conjDGF}) glue as  a germ of \Cex-homeomorphism $\Psi_\G$ which conjugates $\G^\sharp$ to the foliation $\F^\sharp_{\C^\tau}(0)$ obtained by restricting 
 $\F_{\C^\tau}^\sharp$  to the fiber  $M_{\mf u}(0):={\pi_{\mf u}^{-1}(0)}$:
\[
\Psi_{\G}: (M_\G, \mc E_\G)\;
\iso\;( M_{\mf u}(0), \mc E_{\mf u}(0))
\,,
\quad
\Psi_\G(\G^\sharp)=\F^\sharp_{\C^\tau}(0)\,,
\]
with $\mc E_{\mf u}(0):= \mc E_{\mf u}\cap\pi_{\mf u}^{-1}(0)$. The homeomorphism  $\Psi_\G\circ g : \mc E\to  \mc E_{\mf u}(0)$ defines a marking of $\F_{\C^\tau}(0)$, that extends to a marking $(f_{z})_{z}$ of the global family $\un\F_{\C^\tau} $ thanks to property~(\ref{extmarking}) of Remark \ref{prglobalmark}.
Since $S_{\mf u}$ is Stein, up to a biholomorphism over $\C^\tau$, we can assume by classical arguments that $M_{\mf u}^\flat$ is a neighborhood of 
$\{m_0\}\times\C^\tau\subset M_0\times\C^\tau$ and $\pi_{\mf u}^\flat$ is the projection map onto $\C^\tau$.
\\

\textit{-Step 2. }
 Now we also fix $\tilde z\in\C^\tau$ and we consider the germ  $\un\F_{\C^\tau,\tz}$ of $\un\F_{\mbf U}$ at $\tu:=(\tz,d)$ as a deformation of its fiber  $\F_{\C^\tau}(\tz)$. We shall construct a \emph{good trivializing system} for $\un\F_{\C^\tau,\tz}$ in the sense of \cite[Theorem~3.8]{MMSII}, i.e.
a collection $(\Upsilon_{D_{\tz}})_{D_{\tz}}$  of excellent homeomorphism germs
\[\Upsilon_{D_{\tz}}:(M_{\un\F_{\C^\tau,\tz}},D_{\tz})\iso(M_{\F_{\C^\tau}(\tz)}\times\C^\tau,D_{\tz}\times\{\tz\}),\]
where $D_{\tz}$ varies in the set of irreducible components of the exceptional divisor of the reduction of $\F_{\C^\tau}(\tz)$, such that:
\begin{enumerate}[(i)]
\item $\Upsilon_{D_{\tz}}$ is a map over $\C^\tau$ and it is the identity map over $\tz$,
\item $\Upsilon_{D_{\tz}}$ conjugates the foliation $\F^\sharp_{\C^\tau,\tz}$ to the foliation $\F^{\mr{ct}\sharp}_{\C^\tau,\tz}$ obtained after reduction of the constant family $\F^{\mr{ct}}_{\C^\tau,\tz}$,
\item when $D_{\tz}\cap D'_{\tz}$ is either  a nodal singular point or a regular point of $\F_{\C^\tau}(\tz)$, 
the germs of $\Upsilon_{D_{\tz}}$ and $\Upsilon_{D'_{\tz}}$  coincide at $D_{\tz}\cap D'_{\tz}$.
\end{enumerate}
If we restrict the map germ (\ref{conjprunivblup}) to the fiber   over $\tz$, we obtain  a biholomorphism germ $F_\tz$ between the ambient space $M_{\F_{\C^\tau}(\tilde z)}$ of $\F_{\C^\tau}^{\sharp}(\tilde z)$ and the manifold germ $(\pi_{\mf u}^{-1}(\tz), \mc E_{\mf u}\cap \pi_{\mf u}^{-1}(\tz))$. This manifold  is also the manifold germ 
\[
\big(\;M_{\mf u(\tilde z)}, \mc E_{\mf u(\tilde z)}\,\big):=\Big(
\bigcup_{D\in\Ve_{\A_{\mc E}}} U_D\times\{D\} \big/ \sim_{\mf u(\tilde z)},\; 
\bigcup_{D\in\Ve_{\A_{\mc E}}} f(D)\times\{D\} \big/ \sim_{\mf u(\tilde z)}\;\Big)\,,
\]
\noindent defined by the gluing process given by the equivalence relation $\sim_{\mf u(\tilde z)}$ defined by the family:
\[
\mf u(\tilde z):=(\Phi^{\tilde z}_{D,\ge})_{D\in \ge\in\Ed_{\A_{\mc E}}}\,,
\]
\[\Phi^{\tilde z}_{D,\ge} : 
(U_D, s_{\ge})\iso (U_{D'}, s_{\ge})\,,\quad m\mapsto \Phi_{D,\ge}(m,\tilde z)\,,\quad \ge=\langle D,D'\rangle\,.
\]
Clearly $F_\tz$ conjugates $\F_{\C^\tau}^{\sharp}(\tilde z)$ to the foliation 
$\F_{\mf u(\tz)}$ obtained by gluing  $\F^\sharp$  restricted to each $U_D$, i.e.
\[
F_\tz : \big(M_{\F_{\C^\tau}(\tilde z)}, \mc E_{\F_{\C^\tau}(\tilde z)}\big) \iso \big(\;M_{\mf u(\tilde z)}, \mc E_{\mf u(\tilde z)}\,\big)\,,
\quad
F_\tz (\F_{\C^\tau}^{\sharp}(\tilde z))=\F_{\mf u(\tz)}\,.
\]
Let us denote by $\{U_D\times\C^\tau\}\subset M_{\mf u}
$ the image of the canonical embedding $U_D\times \C^\tau\hookrightarrow M_{\mf u}$, and by 
\[g_{D}: \{U_D\times\C^\tau\}\subset M_{\mf u} \longrightarrow U_D\times \C^\tau
\]
the inverse of this embedding. We have the following relations of ``change of charts''
\begin{equation}\label{chcharts}
g_{D'}= \Phi_{D,\langle D,D'\rangle}\circ g_{D}\,.
\end{equation}
Similarly, $\{U_D\}$ denoting the image of the canonical embedding $U_D\hookrightarrow M_{\mf u(\tilde z)}$ and 
\begin{equation*}
g_D^{\tilde z}: \{U_D\}\subset M_{\mf u(\tilde z)}\longrightarrow U_D
\end{equation*}
denoting its inverse, we also have: 
\begin{equation}\label{chchartstz}
g_{D'}^{\tilde z}= \Phi^{\tilde z}_{D,\langle D,D'\rangle}\circ g^{\tilde z}_{D}\,.
\end{equation}

\noindent Notice that  $g_D$ conjugates the foliation $\F_{\mf u}$ restricted to $\{U_D\times \C^\tau\}$ to the constant deformation $(\F^\sharp)_{\C^\tau}^{\mr{ct}}$ of $\F^\sharp$, restricted to $U_D\times \C^\tau$. Similary $g_D^{\tilde z}$ conjugates $\F_{\mf u(\tz)}$ restricted to $\{U_D\}$ to $\F^\sharp$ restricted to $U_D$.  Hence  $g_D^{\tilde z}\times \mr{id}_{\C^\tau}$ conjugates the constant deformation of $\F_{\mf u(\tz)}$ on $\{U_D\}\times \C^\tau$ over $\C^\tau$, denoted by  $(\F_{\mf u(\tz)})_{\C^\tau}^{\mr{ct}}$, to  the constant deformation $(\F^\sharp)_{\C^\tau}^{\mr{ct}}$ of $\F^\sharp$ restricted to $U_D\times \C^\tau$. 
If we write $[U_D\times \C^\tau]:= F^{-1}(\{U_D\times \C^\tau\})$ and $[U_D]:=F_\tz^{-1}(\{U_D\})$, then $F_\tz\times\mr{id}_{\C^\tau}$ conjugates the constant deformation of $\F_{\C^\tau}^{\sharp}(\tz)$ over $\C^\tau$, denoted by $(\F_{\C^\tau}^{\sharp}(\tz))^{\mr{ct}}_{\C^\tau}$, to $(\F_{\mf u(\tz)})_{\C^\tau}^{\mr{ct}}$:
\[
[U_D\times \C^\tau]\stackrel{F}{\iso}\{U_D\times\C^\tau\}\stackrel{g_D}{\iso}U_D\times\C^\tau
\stackrel{(g_D^\tz\times\mr{id}_{\C^\tau})^{-1}}{\stackrel{\sim}{\longrightarrow}}
\{U_D\}\times\C^\tau \stackrel{(F_\tz\times\mr{id}_{\C^\tau})^{-1}}{\stackrel{\sim}{\longrightarrow}}[U_D]\times\C^\tau\,.
\]
\[\phantom{aa}
\F_{\C^\tau}^\sharp \phantom{aa}\mapsto\phantom{aaaa} 
\F_{\mf u}\phantom{aaaa}\mapsto\phantom{a}
(\F^\sharp)_{\C^\tau}^{\mr{ct}}\phantom{aaaa}\mapsto\phantom{aaaa}(\F_{\mf u(\tz)})_{\C^\tau}^{\mr{ct}}
 \phantom{aaaa}\mapsto\phantom{ai} (\F_{\C^\tau}^\sharp(\tz))_{\C^\tau}^{\mr{ct}}
\]
The  homeomorphism 
\[
\Upsilon_D:= (F_\tz\times\mr{id}_{\C^\tau})^{-1}\circ (g_D^\tz\times\mr{id}_{\C^\tau})^{-1} \circ g_D \circ F : [U_D\times \C^\tau]\iso [U_D]\times\C^\tau\,,
\]
is a \Cex-trivialization of $\F^\sharp_{\C^\tau}$ as deformation of $\F_{\C^\tau}^\sharp(\tz)$, i.e. a \Cex-conjugacy from $\F_{\C^\tau}^\sharp$ to 
the constant deformation of $\F^\sharp_{\C^\tau}(\tz)$. The collection $(\Upsilon_D)_{D}$ is a good trivializing system for $\un\F_{\C^\tau,\tz}$, 
after identifying each 
$D\in\Ve_{\A_{\mc E}}$ 
with the irreducible component $D_{\tz}=(g_D^{\tz}\circ F_{\tz})^{-1}(D)\subset [U_D]$ of the exceptional divisor of the reduction of $\F_{\C^\tau}(\tz)$. 

\medskip
\medskip

\textit{-Step 3.} Let us denote here $\F_{\C^\tau}(\tilde z)$ simply by $\F_{\tilde z}$.
To obtain the Kodaira-Spencer map of $\un\F_{\C^\tau,\tz}$ we need first to compute the cocycle 
\[\Upsilon_{D,\ge}:=\Upsilon_{D'}\circ\Upsilon^{-1}_D, \quad \ge=\langle D,D'\rangle\in\Ed_{\A_{\mc E}}.\]
as a germ of biholomorphism $([U_D]\cap[U_{D'}])\times\C^\tau$ at the  point $(D_{\tz}\cap D'_{\tz})\times\{\tz\}$ leaving invariant the constant family $(\F^\sharp_\tz)^\mr{ct}_{\C^\tau,\tz}$.
One easily checks that for each $v\in T_\tz\C^\tau$ the Lie derivative
\[L_v(\mr{pr}_{M_{\F_\tz}}\circ\Upsilon_{D,\ge}):m\mapsto D(\mr{pr}_{M_{\F_\tz}}\circ\Upsilon_{D,\ge})_{(m,\tz)}(0,v)\in T_m M_{\F_\tz}\]
is a well defined and basic vector field for $\F_{\tz}$.
By definition  \cite[\S6.2]{MMSII}, the Kodaira-Spencer map of $\F_{\C^\tau,\tz}$ is
\[\left.\frac{\partial \F_{\C^\tau}}{\partial z}\right|_{z=\tz}:T_{\tz}\C^\tau\to H^1(\A_{\F_{\tz}},\T_{\F_{\tz}}),\quad v\mapsto [(L_v(\mr{pr}_{M_{\F_{\tz}}}\circ\Upsilon_{D,\ge})_{D_\tz,\ge_\tz}],
\]
where  $(D_\tz,\ge_\tz)\in\mc I_{\A_{\F_\tz}}$, $(D,\ge)=f_\tz^{-1}(D_\tz,\ge_\tz)\in\mc I_{\A_{\mc E}}$ and $f_\tz:\mc E\to\mc E_{\F_\tz}$ is the marking of $\F_\tz$ introduced at the end of Step 1.
To see that it is an isomorphism we work in
 the following ``chart'' 
\[\chi:=
(g_{D}^\tz\times\mr{id}_{\C^\tau})\circ (F_\tz\times \mr{id}_{\C^\tau}): [U_D]\times\C^\tau \iso U_D\times\C^\tau\,.
\]
We get
\[
\wt \Upsilon_{D,\ge}:=\chi\circ \Upsilon_{D,\ge}
\circ\chi^{-1}=(g_D^\tz\times \mr{id}_{\C^\tau})\circ ((g_{D'}^{\tz})^{-1}\times \mr{id}_{\C^\tau})\circ g_{D'}\circ g_D^{-1}\,
\]
and, thanks to (\ref{chcharts}) and (\ref{chchartstz}) we obtain
\[\wt \Upsilon_{D,\ge}=(\Phi_{D,\ge}^{\tz}\times \mr{id}_{\C^\tau})^{-1}\circ \Phi_{D,\ge}\,.
\]
Using the explicit  expressions of $\Phi_{D,\ge}$ and $\Phi_{D,\ge}^\tz$, we finally have, writing $z=(z_1,\ldots,z_\tau)\in\C^\tau$ and $\tilde z=(\tilde z_1,\ldots,\tilde z_\tau)$,
\begin{enumerate}[(i)]
\item $\wt \Upsilon_{D,\ge}(m,z)=(\exp(X_{\ge})[z_{\kappa(\ge)}-\tz_{\kappa(\ge)}](m),z)$, if $(D,\ge)\in \mc A$,
\item $\wt \Upsilon_{D,\ge}=\wt \Upsilon_{D',\ge}^{-1}$, if $(D',\ge)\in \mc A$,
\item $\wt \Upsilon_{D,\ge}(m,z)=
(m,z)$, if $(D,\ge),(D',\ge)\notin \mc A$.
\end{enumerate}
% Let us denote here $\F_{\C^\tau}(\tilde z)$ simply by $\F_{\tilde z}$ and consider the marked foliation $\F_{\tilde z}^\diam=(\F_{\tilde z },f_{\tilde z})$.
We deduce the following partial derivatives:
\[
\tilde Y^k_{f(D),f(\ge)}:= \left.\frac{\partial\mr{pr}_{U_D}\circ \wt\Upsilon_{D,\ge}}{\partial z_k}\right|_{z=\tz}
 =\left\{
 \begin{array}{rl}
 X_{\ge}& \hbox{if }\; (D,\ge)\in\mc A \hbox{ and } k=\kappa((D,\ge)),\\
-X_{\ge}&\hbox{if } \; (D',\ge)\in\mc A \hbox{ and } k=\kappa((D',\ge)),\\
0 & \text{otherwise.}
 \end{array}
 \right.
\]
It follows from  Remark 5.10, Proposition 5.12 and Theorem 2.15 of \cite{MMSII} that the cohomological classes $[\tilde Y^1],\ldots,[\tilde Y^\tau]$, associated by the bijection (\ref{kappa}) to the active oriented edges $\mc A$ of $\F^\diam$, form a basis of the vector space $H^1(\A_{\F},\T_{\F})$.

Since $\Upsilon_{D,\ge}=[(g_D^{\tz}\circ F_{\tz})\times\mr{id}_{\C^\tau}]^{-1}\circ\wt\Upsilon_{D,\ge}\circ[(g_D^{\tz}\circ F_{\tz})\times\mr{id}_{\C^\tau}]$, we deduce that
\[Y^k_{f_\tz(D),f_\tz(\ge)}:=\left.\frac{\partial\mr{pr}_{M_{\F_{\tz}}}\circ\Upsilon_{D,\ge}}{\partial z_k}\right|_{z=\tz}=\left\{
 \begin{array}{rl}
X_{\ge}^{\tz}& \hbox{if }\; (D,\ge)\in\mc A \hbox{ and } k=\kappa((D,\ge)),\\
-X_{\ge}^{\tz}&\hbox{if } \; (D',\ge)\in\mc A \hbox{ and } k=\kappa((D',\ge)),\\
0 & \text{otherwise,}
 \end{array}
 \right.\]
 where $X_\ge^{\tz}=(g_D^{\tz}\circ F_{\tz})^*X_\ge$.
 Since $g_D^{\tz}\circ F_{\tz}:[U_D]\to U_D$ is a biholomorphism conjugating $\F_{\tz}^\sharp$ and $\F^\sharp$ we deduce that 
$S_{\F_\tz^\diam}$, defined as in (\ref{SF}), coincides with $S_{\F^\diam}
$ and consequently $\mc A$ is a set of active oriented edges for $\F_\tz^\diam$.
 Hence $[Y^1],\ldots,[Y^\tau]$ form a basis of $H^1(\A_{\F_\tz},\T_{\F_\tz})$ and the Kodaira-Spencer map $\frac{\partial\F_{\C^\tau}}{\partial z}\Big|_{z=\tz}$ is an isomorphism.
By the criterion of universality given in \cite[Theorem 6.7]{MMSII} we conclude that the deformation $(\un\F_{\mbf U,\tilde u},\iota_{\tilde u})$ of $\F_{\tilde u}$ is \Cex-universal.
This achieves the proof of Theorem~\ref{localuniv}.
\end{dem2}

\section{Factorization properties of the locally universal family}\label{Sfact}

\subsection{Local factorization property}
We will now prove  that the global  family  $\un\F_{\mbf U}^{\diam}$ of Theorem~\ref{localuniv} is complete in a similar meaning to that given by  Kodaira-Spencer in \cite{KS} in the context of complex manifolds:

\begin{teo}\label{factolocale}
Let $\F^\diam=(\F,f)$ be a marked by $\mc E^\diam$ foliation of finite type which is a generalized curve and
let $\un\G_P^\diam$  be a  marked global family in $\mr{SL}_{P}(\F^\diam)$. Let us consider  $t_0\in P$ 
 and  $\tu\in{\mbf U}$ such that the marked foliation $\G^\diam_P(t_0)$  is  $\mc C^{\mr{ex}}$-conjugated to the fiber $\F^{\diam}_{\mbf U}(\tu)$ of the marked global family $\un\F^{\diam}_{\mbf U}$ given by Theorem \ref{localuniv}. Then there exists a unique germ of holomorphic map $\lambda:(P,t_0)\to ({\mbf U},\tu)$ such that the germ of $\un\G^\diam_{P}$ at $t_0$ is $\mc C^{\mr{ex}}$-conjugated,  as marked family, to the germ at $t_0$  of $\lambda^\ast\un\F^{\diam}_{\mbf U}$.
%  In particular,
%\[
%\xymatrix{
%(\mbf U,\tilde u)\ar[rrrr]^{\mr{mod}_{\un\F_{\mbf U}^{\diam}}}&&&&\mr{Mod}([\F^\diam])\\ &&&&\\
%(P,t_0)\ar[uurrrr]_{\mr{mod}_{\un\G_P^\diam}}\ar@/_0.5pc/[uu]_{\lambda}
%%\ar@/_2pc/[uu]_{\mu}
%&&&&
%}
%\]
\end{teo}
\begin{dem} Let  $\phi$ be an \Cex-homeomorphism such that    
\[\phi(\G^\diam_P(t_0))=\F^{\diam}_{\mbf U}(\tu)\,.
\] 
We will  denote by  $(\un\G_{P,t_0},{\delta})$  the deformation  of ${\G}_P(t_0)$ over the germ of manifold $(P,t_0)$ defined by the germ of $\un\G_{P}$ at $t_0$ and by the embedding $\delta$ given by the inclusion map of the ambient space of $\G_{P}(t_0)$ in that of $\G_P$. 
According to  Theorem~\ref{localuniv} and Theorem~\ref{phi*}
 the deformation $(\un\F_{\mbf U,\tu},{\iota_\tu})$ is \Cex-universal  and any deformation  
\[
(\un\G_{\mbf U,\tu},\kappa)\in\phi^*([\un\F_{\mbf U,\tu},{\iota_\tu}])\;\in\;\;\mr{Def}_{\G_P(t_0)}^{(\mbf U,\tu)}
\]
is a $\mc C^{\mr{ex}}$-universal deformation of  $\G_{P}(t_0)$.
There exist a holomorphic map germ $\lambda :(P,t_0)\to( \mbf U,\tu)$ and a $\mc C^{\mr{ex}}$-conjugacy $\Phi_{P,t_0}$ from the deformation $(\un\G_{P,t_0},\delta)$ to $\lambda^\ast(\un\G_{\mbf U,\tu},\kappa)$.
By definition of $\phi^*$, the associated  families   $\un\G_{\mbf U,\tu}$ and  $\un{\F}_{\mbf U,\tu}$ are $\mc C^{\mr{ex}}$-conjugated  as germs of families over $(\mbf U,\tu)$,
by a $\mc C^{\mr{ex}}$-homeomorphism $\Phi_{\mbf U,\tu}$ which is equal to $\phi$ over $\tu$.
Since $\phi$ is compatible with the markings of $\G^\diam_P(t_0)$ and $\F_{\mbf U}^\diam(\tu)$, it follows from Remark~\ref{prglobalmark} that 
the homeomorphism 
$\lambda^\ast \Phi_{\mbf U,\tu}\circ \Phi_{P,t_0}$ that conjugates the associated family  
$\un\G_{P,t_0}$ to  $\lambda^\ast \un\F_{\mbf U,\tu}$, is 
compatible with the markings of the  families
$\un\G_{P,t_0}^\diam$ and $\lambda^\ast\un\F_{\mbf U,\tu}^{\diam}$. 
The uniqueness of $\lambda$ results from the following lemma 
because local $\mc C^{\mr{ex}}$-conjugacy between $\un\G_{P,t_0}^\diam$ and $\lambda^\ast\un\F_{\mbf U,\tu}^{\diam}$ implies  $\mc C^{\mr{ex}}$-conjugacy of the fibers $\G^\diam_P(t)$ and $\F^\diam_{\mbf U}(\lambda(t))$  and consequently the equality of the modular maps $\mr{mod}_{\un\G_{P,t_0}^\diam}$ and $\mr{mod}_{\lambda^*\un\F^\diam_{\mbf U,\tilde u}}=\mr{mod}_{\un\F_{\mbf U,\tilde u}^\diam}\circ\lambda$.
\end{dem}

\begin{lema}\label{liftmodP} 
Let $\un\G^\diam_P$ be a marked global  family in $\mr{SL}_P(\F^\diam)$ over a connected complex manifold $P$.   
Two holomorphic liftings $\lambda, \mu:P\to \mbf U$  of the moduli map of $\un\G_{P}^\diam$ through the moduli map of $\un\F_{\mbf U}^{\diam}$,
\begin{equation}\label{lifting}
{\mr{mod}_{\un\F_{\mbf U}^{\diam}}}\circ\lambda={\mr{mod}_{\un\F_{\mbf U}^{\diam}}}\circ\mu=\mr{mod}_{\un\G_P^\diam}\,,
\end{equation}
\[
\xymatrix{
\mbf U\ar[rrrr]^{\mr{mod}_{\un\F_{\mbf U}^{\diam}}}&&&&\mr{Mod}([\F^\diam])\\ &&&&\\
P\ar[uurrrr]_{\mr{mod}_{\un\G_P^\diam}}\ar@/_0.5pc/[uu]_{\lambda}\ar@/_2pc/[uu]_{\mu}&&&&
}
\]
coincide as soon as they take the same value at some  point $t_0\in P$. 
\end{lema} 
\begin{dem} 
Let $d\in\mbf D$ be the image of $[\G_P^\diam(t_0)]$ by the morphism $\beta$ in (\ref{excsequmod1}). 
We write $\lambda=(\lambda_1,\lambda_2)$ and $\mu=(\mu_1,\mu_2)$ with $\lambda_1,\mu_1:P\to\C^\tau$ and $\lambda_2,\mu_2:P\to\mbf D$.
Since $\lambda,\mu$ are holomorphic, $P$ is connected and $\mbf D$ is totally discontinuous we deduce that $\lambda_2$ and $\mu_2$ are constant equal to $d$.
It follows from relations (\ref{propLambda}) and (\ref{lifting}) that for any $t\in P$ we have $\Lambda(\lambda_1(t))\cdot \zeta(d)=\Lambda(\mu_1(t)) \cdot \zeta(d)$.
According to the exact sequence (\ref{excsequmod1}), there is $N_t\in \Z^p$ such that $\lambda_1(t)-\mu_1(t)=\alpha(N_t)$. 
The following sets
\[
K_N:=\{t\in P\;;\; \alpha(N_t)=\alpha(N)\}\,,
\quad
N\in\Z^p\,,
\]
are closed analytic subsets of $P$ given by the global equations $\lambda_1(t)-\mu_1(t)=\alpha(N)$. All of them cannot be proper subsets of $P$, because $P=\cup_{N\in\Z^p}K_N$. Therefore there exists $N_0\in\Z^p$ such that $\lambda_1(t)-\mu_1(t)=\alpha(N_0)$ for any $t\in P$. As $\lambda(t_0)=\mu(t_0)$, we have $\alpha(N_0)=0$, which ends the proof.
\end{dem}

\begin{cor}\label{BetaLambdacCt} 
Let $\F^\diam$ be a marked by $\mc E^\diam$ foliation of finite type which is a generalized curve.
For any marked  global family $\un\G_P^\diam\in \mr{SL}_{P}(\F^\diam)$ over a connected manifold~$P$ the map $\beta\circ \mr{mod}_{\un\G_{P}^\diam}:P\to \mbf D$ is constant, where $\beta : \mr{Mod}([\F^\diam])\to \mbf D$ still denotes the last group morphism in the exact sequence (\ref{excsequmod1}).
\end{cor}
\begin{dem}
It suffices to prove that the map $\beta\circ \mr{mod}_{\un\F_{\mbf U}^{\diam}}$ is locally constant. Let $t_0$ be  a point in $P$. There is $\tu\in \mbf U$ such that $[\F^{\,\diam}_{\mbf U}(\tu)]=[\G^\diam_P(t_0)]\in \mr{Mod}([\F^\diam])$. Theorem~\ref{factolocale} provides  a holomorphic map germ $\lambda=(\lambda_1,\lambda_2) : (P,t_0)\to (\mbf U, \tu)=(\C^\tau\times\mbf D,(\tilde z,d))$ such that, according to (\ref{propLambda}),  for $t\in P$ close to $t_0$, we have:
\[
\mr{mod}_{\un{\G}_P^\diam}(t)=\mr{mod}_{\un\F_{\mbf U}^{\diam}}(\lambda(t))=\Lambda(\lambda_1(t))\cdot \zeta(\lambda_2(t))
\;\in\;\Lambda(\C^\tau)\cdot\zeta(d)=\beta^{-1}(d)\,,
\]
using that $\lambda_2:P\to\mbf D$ is constant equal to $d$ as we have already remarked in the proof of Lemma~\ref{liftmodP}.
\end{dem}

\subsection{Global factorization property}
We are interested now in factorizing up to $\mc C^{\mr{ex}}$-conjugacy marked  global families through the marked global family $\un\F_{\mbf U}^{\diam}$ provided by Theorem~\ref{localuniv}.  
If $\un\F_Q=(M,\pi,\theta,\F_Q)$ is a global family of foliations over $Q$ for which there exists $\lambda:Q\to\mbf U$ such that $\un\F_Q=\lambda^*\un\F_{\mbf U}$, then the germ of the ambient space
$M$ along $\theta(Q)$ is biholomorphic to the product $(M(u_0),\theta(u_0))\times Q$. 
To avoid obstructions to such factorizations on the ambient space we consider a weaker conjugacy relation.
\begin{defin}\label{locexcong}
Two  equisingular (resp. marked equisingular) global families over a manifold $Q$ are \emph{locally \Cex-conjugated}  
if their germs at any point of $Q$ are \Cex-conjugated  as families (resp. as marked families, see Section~\ref{sec-equi-global}).
\end{defin}
The object of this section is to prove the following theorem of factorization up to local \Cex-conjugacy. 
\begin{teo}\label{inivfaible}
Let $\F^\diam$ be a marked by $\mc E^\diam$ foliation of finite type which is a generalized curve and let $\un\F^{\diam}_{\mbf U}$  be the marked equisingular global family given by Theorem~\ref{localuniv}.
 Let $P$ be a connected manifold satisfying  $H_1(P,\mb Z)=0$ and let $\un\G_{P}^\diam$ be a global family in $\mr{SL}_P(\F^\diam)$.  Then for any $t_0\in P$ and $(\tz,d)\in {\mbf U}$ such that the marked foliations $\G^\diam_P(t_0)$ and $\F_{\mbf U}^{\diam}(\tz,d)$ are \Cex-conjugated, there exists a unique holomorphic map $\lambda:P\to {\mbf U}$ satisfying  $\lambda(t_0)=(\tz,d)$, such that the marked global families $\un\G_P^\diam$ and $\lambda^\ast\un\F^{\diam}_{\mbf U}$ are locally \Cex-conjugated.
\end{teo}
\begin{dem}
{\it - Step 1: Construction of $\lambda$. } According to Corollary \ref{BetaLambdacCt},  $\mr{mod}_{\un\G^\diam_{P}}$ takes values in $\beta^{-1}(d)$. Thus for any $t\in P$, there exist  $z_t\in\C^\tau$ and a \Cex-homeomorphism $\phi_t$ such that $\phi_t(\G_{P}(t))= \F_{\mbf U}^{\diam}(z_t,d)$. Let us denote by $\un\G^\diam_{P,\,t}$ the germ of $\un \G^\diam_P$ at $t$. According to Theorem~\ref{factolocale}  there exists a holomorphic map germ $\lambda_t: (P,t)\to (\mbf U,(z_t,d))$ such that   $\un\G^\diam_{P,t}$ and $\lambda_t^\ast\un\F_{\mbf U}^{\diam}$ are \Cex-conjugated, as germs of families. Therefore there exist an open covering $(V_i)_{i\in I}$, $I \subset \N$, of $P$ and holomorphic maps $\lambda_i:V_i\to \C^\tau$ such that the restriction  of $\un \G^\diam_{P}$ to $V_i$ are \Cex-conjugated to $(\lambda_i,d)^\ast\un\F_{\mbf U}^{\diam}$. Thus we have
\begin{equation*}\label{modlambdai}
\mr{mod}_{\un\F_{\mbf U}^{\diam}} \circ (\lambda_i,d)=\mr{mod}_{\un\G^\diam_{P}}|_{V_i}\,.
\end{equation*}
We can also require that this covering is locally finite and that the open sets $V_i$ and $V_i\cap V_j$, $i,j\in I$,  are connected. When $V_i\cap V_j$ is non empty, the restrictions of $(\lambda_i,d)$ and $(\lambda_j,d)$ to this open set are two factorizations of the moduli map  of $\un\G_P^\diam$ through $\mr{mod}_{\un\F_{\mbf U}^{\diam}}$. Fixing a point $t_{ij}$ in $ V_i\cap V_j$, we have:    
\[
[\F_{\mbf U}^{\diam}(\lambda_i(t_{ij}),d)]=
[\F_{\mbf U}^{\diam}(\lambda_j(t_{ij}),d)]=
[\G_{P}^\diam(t_{ij})]\,.
\]
The relation (\ref{propLambda}) gives 
$\Lambda(\lambda_i(t_{ij})) \cdot \zeta(d) = 
\Lambda(\lambda_j(t_{ij})) \cdot \zeta(d)$;  thus $(\lambda_j(t_{ij})-\lambda_i(t_{ij}))$ belongs to the kernel of $\Lambda$ and there exist $N_{ij}\in \Z^p$ such that $\lambda_j(t_{ij})-\lambda_i(t_{ij})=\alpha(N_{ij})$.  As by assumption we have: $H_1(P,\Z)=0$, the \v{C}ech cohomology group $H^1(P,\Z^p)$ is trivial and there exist $N_i\in\Z^p$, ${i\in I}$, such that $N_j-N_i=N_{ij}$ as soon as $V_i\cap V_j\neq \emptyset$. Notice that the maps $(\alpha(N_i)+\lambda_i,d):V_i\to\mbf U$ are again liftings of $\mr{mod}_{\un\G_{P}^\diam}|_{V_i}$ through $\mr{mod}_{\un\F_{\mbf U}^{\diam}}$. Indeed we have:

\begin{align*}
\mr{mod}_{\un\F_{\mbf U}^{\diam}} \circ (\alpha(N_i)+\lambda_i,d)=&
\Lambda(\alpha(N_i)+\lambda_i) \cdot \zeta(d)\\
=&
\Lambda(\lambda_i) \cdot \zeta(d)
=\mr{mod}_{\un\F_{\mbf U}^{\diam}} \circ (\lambda_i,d) = \mr{mod}_{\un\G_P^\diam}|_{V_i}\,.
\end{align*}
Since $\alpha(N_i)+\lambda_i(t_{ij})=\alpha(N_j)+\lambda_j(t_{ij})$, 
thanks to Lemma~\ref{liftmodP} and the connectedness of $V_i\cap V_j$, the maps $\alpha(N_i)+\lambda_i$, $i\in I$, glue as a global holomorphic map
\begin{equation*}\label{globliftmodP}
\lambda:P\longrightarrow \C^\tau\times \{d\}\subset \mbf U\,,
\quad 
 \mr{mod}_{\un\F_{\mbf U}^{\diam}} \circ \lambda=\mr{mod}_{\un\G^\diam_{P}}\,.
\end{equation*}
{\it Step 2: Properties of $\lambda$. } First we notice that  for any   $N\in\Z^p$ we also have the equality
$ \mr{mod}_{\un\F_{\mbf U}^{\diam}} \circ (\alpha(N)+\mr{pr}_{\C^\tau}\circ\lambda,d)=\mr{mod}_{\un\G^\diam_{P}}
$. Consequently we can assume that  $\lambda(t_0)=(\tz,d)$. On the other hand,  the global families    $\un\G^\diam_P|_{V_i}$ and $(\lambda_i,d)^\ast\un\F_{\mbf U}^{\diam}$  being  \Cex-conjugated,  the local \Cex-conjugacy between $\un\G_P^\diam$ and $\lambda^\ast\un\F_{\mbf U}^{\diam}$ results from the lemma below.
\end{dem}

\begin{lema}
If $\mu:(Q, u_0)\to (\C^\tau,z_0)$ is a holomorphic map germ and $N\in\Z^p$, then the germs at $u_0$ of the marked  families $(\mu,d)^\ast\un\F_{\mbf U}^{\diam}$ and $(\alpha(N)+\mu,d)^\ast\un\F_{\mbf U}^{\diam}$ are \Cex-conjugated.
\end{lema}
\begin{dem}
Let us denote by $\un\F_{\mbf U,(z,d)}^{\diam}$ the germ of $\un\F_{\mbf U}^{\diam}$ at $(z,d)$ considered as a deformation of the foliation $\F^{\diam}_{\mbf U}(z,d)$, the embedding map being the inclusion $M_0\times\{(z,d)\}\hookrightarrow M_0\times \mbf U$. 

As $(\alpha(N)+\mu,d)^\ast\un\F_{\mbf U}^{\diam}=(\mu,d)^\ast(\Delta^\ast\un\F_{\mbf U}^{\diam})$, with 
\[\Delta:(\mbf U, (z_0,d))\longrightarrow (\mbf U, (\alpha(N)+z_0,d))\,,
\quad
\Delta(z,d):= (\alpha(N)+z,d)\,,
\]
it suffices to see that  $\un\F_{\mbf U,(z_0,d)}^{\diam}$ is \Cex-conjugated  to $\Delta^\ast(\un\F_{\mbf U, (\alpha(N)+z_0,d)}^{\diam})$ as a family. To lighten the text let us  write 
\[
\F^\diam_0:=\F_{\mbf U}^{\diam}(z_0,d)
\quad\hbox{and}\quad
\F^\diam_N:=\F_{\mbf U}^{\diam}(\alpha(N)+z_0,d)=(\Delta^*\F_{\mbf U}^{\diam})(z_0,d)\,.
\]
There exists a \Cex-homeomorphism $\phi$ such that  $\phi(\F^\diam_N)=\F^\diam_0$. 
Let $\un{\mc K}^\diam_{\mbf U,(z_0,d)}$ be a deformation of $\F^\diam_N$ over the germ of manifold $(\mbf U,(z_0,d))$ that belongs  to the class $\phi^*([\un\F_{\mbf U,(z_0,d)}^{\diam}])$. 
According to Theorem~\ref{localuniv} the deformations $\un\F_{\mbf U,(z_0,d)}^{\diam}$ and   $\un\F_{\mbf U, (\alpha(N)+z_0,d)}^{\diam}$ are   \Cex-universal; it follows from Theorem~\ref{phi*}
 and Remark~\ref{rempreliminaires} that the deformation $\un{\mc K}_{\mbf U,(z_0,d)}^\diam$ of $\F_N^\diam$ is \Cex-universal. On the other hand, since $\Delta$ is a biholomorphism, the deformation $\Delta^\ast(\un\F_{\mbf U, (\alpha(N)+z_0,d)}^{\diam})$ of $\F_N^\diam$ is also \Cex-universal over the same parameter space $(\mbf U,(z_0,d))$, again by Remark~\ref{rempreliminaires}.
 By uniqueness of \Cex-universal deformations, $\un{\mc K}^\diam_{\mbf U,(z_0,d)}$  and $\Delta^\ast(\un\F_{\mbf U, (\alpha(N)+z_0,d)}^{\diam})$
 are \Cex-conjugated deformations of $\F_N^\diam$. We end the proof by noting that by definition of $\phi^*([\un\F_{\mbf U,(z_0,d)}^{\diam}])$, the families $\un{\mc K}^\diam_{\mbf U,(z_0,d)}$ and $\un\F_{\mbf U,(z_0,d)}^{\diam}$ are \Cex-conjugated.
\end{dem}

Now, we consider a weaker notion of conjugacy requiring the equality of moduli maps, in other words, the \Cex-conjugacy fiber by fiber for each value of the parameter.

\begin{teo} \label{310}
Let $\F^\diam$ be a marked by $\mc E^\diam$ foliation of finite type which is a generalized curve and let $\un\F^{\diam}_{\mbf U}$  be the marked equisingular global family given by Theorem~\ref{localuniv}.
%If $P$ is a connected manifold such that $H_1(P,\Z)=0$,  then: 
 \begin{enumerate}%[(1)]
 \item\label{unicityfactmap} 
 Assume that $P$ is a connected manifold such that $H_1(P,\Z)=0$, then the moduli map of any marked global family $\un\G^\diam_P\in\mr{SL}_P(\F^\diam)$ factorizes through the moduli map of $\un\F_{\mbf U}^{\diam}$. 
More precisely, 
for any $t_0\in P$ and $(\tz,d)\in\mbf U$ such that $\G^\diam_P(t_0)$ is \Cex-conjugated to $\F_{\mbf U}^{\diam}(\tz,d)$, 
there is a unique holomorphic map $\lambda:P\to\mbf U$ satisfying $\mr{mod}_{\un\G^\diam_P}= \mr{mod}_{\un\F^{\diam}_{\mbf U}}\circ \lambda$ and $\lambda(t_0)=(\tilde z,d)$.
\item\label{redondancemcg} The non-marked foliations $\F_{\mbf U}(u_1)$ and $\F_{\mbf U}(u_2)$ are \Cex-conjugated if and only if there is $\dot\varphi\in\mbf I_{\F^\diam}$ such that $\dot\varphi\star[\F_{\mbf U}^\diam(u_1)]=[\F_{\mbf U}^\diam(u_2)]$, see Definition~\ref{IF}. 
 \end{enumerate}
\end{teo}
Notice that a priori the uniqueness of $\lambda$  stated in assertion (\ref{unicityfactmap}), is a stronger property than that given by Theorem~\ref{inivfaible} because the property that the marked global families $\un\G_P^\diam$ and $\lambda^*\un\F_{\mbf U}^\diam$ are locally $\mc C^{\mr{ex}}$-conjugated implies that $\mr{mod}_{\un\G^\diam_P}= \mr{mod}_{\un\F^{\diam}_{\mbf U}}\circ \lambda$. In Theorem~\ref{teo91} we will see the equivalence of these two properties for a family of generalized curves of finite type.

\begin{obs}\label{identification}
Thanks to the exact sequence  (\ref{excsequmod1}) we have an action $\ast$ of $\Z^p$ on $\mbf U=\C^\tau\times\mbf D$ given by $N\ast(z,d)=(z+\alpha(N),d)$.
For each section $\zeta:\mbf D\to\mr{Mod}([\F^\diam])$ we have a family $\un\F_{\mbf U}^\diam$ over $\mbf U$ and
 an identification of $\mbf U/\Z^p=(\C^\tau/\alpha(\Z^p))\times\mbf D$ with $ \mr{Mod}([\F^\diam])$ by the map $([z],d)\mapsto \Lambda(z)\cdot \zeta(d)$. 
Using this identification we obtain an action that we still denote by $\star$ 
of the discrete group $\mbf I_{\F^\diam}$ on the quotient $\mbf U/\Z^p$ such that 
 $\F_{\mbf U}(u_1)$ and $\F_{\mbf U}(u_2)$ are \Cex-conjugated if and only if there exists $\dot\varphi\in\mbf I_{\F^\diam}$ such that $\dot\varphi\star(\Z^p*u_1)=\Z^p*u_2$.
\end{obs}

\begin{proof}[Proof of Theorem~\ref{310}]
The existence of the factorization $\lambda$ of $\mr{mod}_{\un\G^\diam_P}$ in assertion (\ref{unicityfactmap}) 
follows from Theorem~\ref{inivfaible} and its uniqueness under the assumption $\lambda(t_0)=(\tilde z,d)$ is given by Lemma~\ref{liftmodP}. 
Assertion~(\ref{redondancemcg}) follows from Proposition~\ref{orbit-fiber}.
\end{proof}

\begin{cor}
If $P$ is a connected compact manifold such that $H_1(P,\Z)=0$  then any marked global family $\un\G^\diam_P\in\mr{SL}_P(\F^\diam)$ is locally $\mc C^{\mr{ex}}$-trivial, and a fortiori the 
topological class of $\G^\diam_P(t)$, $t\in P$, is constant. 
\end{cor}

\begin{proof}[Proof of Theorem~\ref{C}]
Assertion (\ref{C0}) corresponds to properties (\ref{formuletau}) and (\ref{expliciterD}) of the exact sequence (\ref{excsequmod1}) stated in Section~\ref{Suniv}.
Property (\ref{thmmod}) of Theorem~\ref{localuniv} implies assertion (\ref{C1}) of Theorem~\ref{C}, while assertion (\ref{C2}) of Theorem~\ref{C} is stated in Theorem~\ref{inivfaible}.
\end{proof}

\section{Topological equivalences for families and deformations}\label{SectionWeakTop} 
We will compare for global families and for germs of deformations, the \Cex-conjugacy relation to a weaker conjugacy relation defined as  the topological conjugacy before reduction, on each fiber of the family, without requiring the continuous dependence on the parameters of the conjugating homeomorphisms.

\subsection{Tame foliations}\label{tame}

Until now the only hypothesis that we have made on the germs of generalized curve foliations is that of being of finite type. Under this hypothesis, which is Krull generic \cite{MS}, we have obtained, for the equivalence relation \Cex, complete families whose modular map is surjective.
In order to obtain the same result for the equivalence relation  $\mc C^0$ we must add a combinatorial assumption on the exceptional divisor $\mc E_\F$ and a dynamical assumption on the transverse structure of the foliation $\F$.
For that let us denote by $\mc E_\F^d$ the union of irreducible components of the exceptional divisor $\mc E_\F$ which are dicritical and by $\mc N\mc C_\F$ the set of singular points of $\mc E_\F$, called \emph{nodal corners}, where the Camacho-Sad index of $\F^\sharp$ is a strictly positive real number. Let us consider the following two conditions:

\medskip

\begin{enumerate}
\item[(NC)] {\it No Chain}: 
the closure of each connected component of $\mc E_\F\setminus\mc E_\F^d$ contains an irreducible component $D$ with $\mr{card}(D\cap\mr{Sing}(\F^\sharp))\neq 2$.
%, i.e. there is no connected component of $\overline{\mc E_\F\setminus\mc E_\F^d}$ as in Figure~\ref{TC1}.\\

\item[(TR)] {\it Transverse Rigidity}: 
if the closure of a connected component of $\mc E_\F\setminus(\mc E_\F^d\cup\mc N\mc C_\F)$
contains an irreducible component with at least $3$ singular points of $\F^\sharp$,  it also contains an irreducible component whose holonomy group for the  foliation $\F^ \sharp$ is {topologically rigid}, for instance unsolvable, cf. \cite{nakai,Rebelo}.\\
\end{enumerate}

%\begin{figure}[h]
%\begin{tikzpicture}
%\node at (6.3,0.6) {$s'$};
%\node at (10.3, 0.6) {$s''$};
%\draw (6.3,0.28) node {$\bullet$};
%\node at (8.35,0.1) {$\cdots$};
%\draw (10.3,0.27) node {$\bullet$};
%\draw (6.785,0.16) node {$\bullet$};
%\draw (9.79,0.16) node {$\bullet$};
%\node at (8.35,0.1) {$\cdots$};
%\draw  (7,0) arc [radius=1, start angle=45, end angle= 130];
%\draw  (8,0) arc [radius=1, start angle=45, end angle= 130];
%\draw  (10,0) arc [radius=1, start angle=45, end angle= 130];
%\draw  (11,0) arc [radius=1, start angle=45, end angle= 130];
%\end{tikzpicture}
%\caption{The only situation excluded by Condition (NC). 
%Every divisor is non-dicritical and the elements of $\mr{Sing}(\F^\sharp)$ are $s'$, $s''$ and the intersection points of the divisors;
%dicritical components may intersect any component.
%}\label{TC1}
%\end{figure}
Condition (NC) is technical and, as for the generalized curve condition, only depends on a  finite order jet of the differential form defining $\F$. In the presence of chains, \Ctop-classification must be approached differently and it will  depend on open questions about the topology of Cremer biholomorphisms in one complex variable. Property~(TR)  is satisfied for a dense open set for the Krull topology  of differential $1$-forms fulfilling condition (NC), cf. \cite{LeFloch}. 
The following theorem, first proven in \cite{MM2} with additional assumptions, then generalized in \cite[Theorem 11.4]{MMS} using results of \cite{Loic1}, justifies these two hypothesis:
\begin{teo}[{\cite[Theorem A]{MMS}}]\label{C0 = Cex}
Two germs of generalized curves foliations $\F$ and $\G$ satisfying (NC) and (TR) are $\mc C^0$-conjugated if and only if they are \Cex-conjugated.
\end{teo}

\begin{defin}
A germ of singular foliation is called \emph{tame} if it is a generalized curve of finite type satisfying conditions (NC) and (TR).
\end{defin}

\begin{obs}
For a global equisingular family over a connected manifold, properties (NC), (TR) and being of finite type are satisfied by any fiber as soon as they are satisfied by one fiber.
\end{obs}

\begin{proof}[Proof of Theorem~\ref{main}]
We mark $\F$ by $\mc E_\F$ using the identity map; to obtain $\un\F_{\mbf U}$ we apply Theorem~\ref{C} to $\F^\diam=(\F,\mr{id}_{\mc E_\F})$, that also provides a marking $(f_u)_{u\in\mbf U}$ on $\un\F_{\mbf U}$.

We begin by proving assertion (\ref{modsurj}). 
Since $\G$ has the same SL-type as $\F$ there exists a homeomorphism $\varphi:\mc E_\F\to\mc E_\G$ satisfying properties (SL1)-(SL3). We consider the marked by $\mc E_\F^\diam=(\mc E_\F,\mr{Sing}(\F^\sharp),\cdot)$ foliation $\G^\diam=(\G,\varphi)$ which has the same marked SL-type as $\F^\diam$. By assertion (1) of Theorem~\ref{C} there exists $u_0\in\mbf U$ such that $\G^\diam$ is \Cex-conjugated to $\F_{\mbf U}^\diam(u_0)$. A fortiori, $\G$ is \Ctop-conjugated to $\F_{\mbf U}(u_0)$.

Let us now prove assertion (\ref{globalfactnonmark}). As $\G_P(t_0)$ is \Ctop-conjugated to $\F_{\mbf U}(u_0)$,
by Theorem~\ref{C0 = Cex} there is a \Cex-conjugacy  $\phi:\F_{\mbf U}(u_0)\to\G_P(t_0)$.
The composition $g_{t_0}:=\phi^\sharp\circ f_{u_0}:\mc E_\F\to\mc E_{\G_P(t_0)}$ of the lifting of $\phi$ through the reduction maps and the marking of $\F_{\mbf U}(u_0)$ defines a marking of $\G_P(t_0)$ such that $\G_P^\diam(t_0)$ is \Cex-conjugated to $\F_{\mbf U}^\diam(u_0)$ by $\phi^{-1}$. Since $P$ is simply connected, by assertion (\ref{extmarking}) of Remark~\ref{prglobalmark}, the marking $g_{t_0}$ extends to a marking $(g_t)_{t\in P}$ of the global family $\un\G_P$. We apply assertion~(\ref{C2}) of Theorem~\ref{C} to $\un\G_P^\diam=(\un\G_P,(g_t)_{t\in P})$ and we obtain a (unique) holomorphic map $\lambda:P\to\mbf U$ such that $\lambda(t_0)=u_0$ and for any $t\in P$ the germs of marked families $\un\G^\diam_{P,t}$ and $\lambda^*\un\F_{\mbf U,\lambda(t)}^\diam$ over the germ of manifold $(P,t)$ are \Cex-conjugated. A fortiori, the germs of families
 $\un\G_{P,t}$ and $\lambda^*\un\F_{\mbf U,\lambda(t)}$ are \Ctop-conjugated.

Redundancy property (\ref{redondance}) in Theorem~\ref{main} follows from assertion (\ref{redondancemcg}) in Theorem~\ref{310} and Remark~\ref{identification} taking into account that \Cex-conjugacy and \Ctop-conjugacy are equivalent for tame foliations, see Theorem~\ref{C0 = Cex}.
\end{proof}

\subsection{Weak and strong conjugacies of families}\label{weak-strong}

In this section we will prove Theorem~\ref{conjug} of the introduction. Before that, we state a marked version of that result in which the hypothesis are weaker.

\begin{teo}\label{teo91}
Let $\un\F^\diam_Q$ and $\un\G_Q^\diam$ be marked by $\mc E^\diam$ equisingular global families of foliations over a complex manifold $Q$, whose fibers are generalized curves of finite type. The following properties are equivalent
\begin{enumerate}
\item  for any $u\in Q$ the marked foliations  $\F^\diam_{Q}(u)$ and $\G^\diam_Q(u)$ are \Cex-conjugated,
\item the marked global families $\un\F^\diam_Q$ and $\un\G^\diam_Q$  are  locally \Cex-conjugated.
\end{enumerate} 
\end{teo}

\begin{proof}
The implication $(2)\Rightarrow(1)$ is trivial. To prove the converse we can assume that $Q$ is connected and simply connected.
Let us fix a fiber $\F^\diam:=\F_Q^\diam(\tu)$, $\tu\in Q$.  According to the connectedness of $Q$ and Remark~\ref{extsl},
each $\G^\diam_Q(u)$, $u\in Q$,  belongs to $\mr{SL}(\F^\diam)$, see Definition~\ref{deftypeSL}. By assertion (1) we have the equality
\begin{equation*}\label{equmodmaps}
\mr{mod}_{\un\G_Q^\diam}=\mr{mod}_{\un\F^\diam_Q} : Q\longrightarrow \mr{Mod}([\F^\diam])\,.
\end{equation*}
Let $\un\F_{\mbf U}^{\diam}$ be 
the marked global family given by Theorem~\ref{localuniv}.
Let us consider
 $(\tz,d)\in \mbf U$ such that $\F_Q^\diam(\tu)$ is \Cex-conjugated to $\F_{\mbf U}^{\diam}(\tz,d)$.
Since $\F_Q^\diam(\tu)$ and $\G_Q^\diam(\tu)$ are \Cex-conjugated,
Theorem \ref{inivfaible} provides  holomorphic maps $\lambda,\mu: Q\to\C^\tau\times\{d\}\subset\mbf U$ satisfying $\lambda(\tu)=\mu(\tu)=(\tz,d)$, such that 
$\un\F^\diam_Q$ is locally \Cex-conjugated to $\lambda^\ast\un\F_{\mbf U}^{\diam}$ and $\un\G^\diam_Q$ is locally \Cex-conjugated to $\mu^\ast\un\F_{\mbf U}^{\diam}$. We thus have:
\[
\mr{mod}_{\un\F_{\mbf U}^{\diam}}\circ\lambda=\mr{mod}_{\un\F^\diam_Q}=\mr{mod}_{\un\G^\diam_Q}=\mr{mod}_{\un\F_{\mbf U}^{\diam}}\circ\mu
\,.
\]
Consequently $\lambda$ and $\mu$  are two liftings of the map $\mr{mod}_{\un\G^\diam_Q}=\mr{mod}_{\un\F^\diam_Q}$ through the map $\mr{mod}_{\un\F_{\mbf U}^{\diam}}$,  which coincide at the point $\tu$. It follows from the uniqueness in assertion~(\ref{unicityfactmap}) of Theorem~\ref{310} that $\lambda=\mu$. Therefore $\un\F_Q^\diam$ and $\un\G^\diam_Q$ are locally \Cex-conjugated, since they are  both locally \Cex-conjugated to $ \lambda^\ast\un\F_{\mbf U}^{\diam}=\mu^\ast\un\F_{\mbf U}^{\diam}$.
\end{proof}

Now we will use Theorem~\ref{teo91} to prove Theorem~\ref{conjug} of the introduction.

\begin{dem2}{of Theorem~\ref{conjug}}
Thanks to Theorem~\ref{C0 = Cex}, assertions~(\ref{weakconj}) and~(\ref{exconj})  are equivalent.
The implication $(\ref{strongconj})\Longrightarrow (\ref{weakconj})$ is trivial. To prove $(\ref{exconj})\Longrightarrow (\ref{strongconj})$
let us fix a point $\tu$ in $Q$ and a marking $f_{\tu}: \mc E\iso \mc E_{\F_Q(\tu)}$ of the fiber $\F_Q(\tu)$
by an appropriate marked divisor~$\mc E^\diam$. 
By restricting both families to a suitable neighborhood of $\tu$ we may assume that $Q$ is connected and simply connected. 
Thanks to (\ref{extmarking}) in Remark \ref{prglobalmark}, $f_\tu$ extends to a marking $(f_u)_{u\in Q}$ of the global family $\un\F_Q$ and we will write:
\[\un\F_Q^\diam:=(\un\F_Q,(f_u)_{u\in Q})\quad\text{and}\quad \F^\diam:=(\F_Q(\tu), f_{\tu})\,.\]
According to Theorem~\ref{teo91}, 
in order to obtain assertion (\ref{strongconj}) 
it only remains to prove the existence of a marking $(G_u)_{u\in Q}$ of $\un\G_Q$ such that for each $u\in Q$ the marked foliation $\G_Q^\diam(u):=(\G_Q(u),G_u)$ is \Cex-conjugated to $\F^\diam_Q(u)$:
\begin{equation}\label{goodmarking}
[(\G_Q(u), G_u)]=[(\F_Q(u),f_u)]\in\mr{Mod}([\F^\diam])\,,\quad
u\in Q\,.
\end{equation}
For this, we choose for each $u\in Q$ a \Cex-conjugacy 
\[
{\phi}_u:(M(u),\theta(u))\iso(N(u),\vartheta(u))\,,
\quad
{\phi}_u(\F_Q(u))=\G_Q(u)\,,
\]
and we denote by ${\phi}^\sharp_u : (M_{\F_Q(u)},\mc E_{\F_Q(u)})\to (M_{\G_Q(u)},\mc E_{\G_Q(u)})$  the germ of homeomorphism obtained by lifting it through the reduction of singularities of $\F_Q(u)$ and $\G_Q(u)$.  We endow $\un\G_Q$ with a marking by $\mc E^\diam$
\[
g_u : \mc E\to\mc E_{\G_Q(u)}\,,\quad u\in Q\,,
\] 
obtained thanks to Remark \ref{prglobalmark} by extending the marking ${\phi}_\tu^\sharp\circ f_{\tu}$ of $\G_Q(\tu)$. 
We also consider the following pre-marking of $\un\G_Q$:
\[
{\phi}^\sharp_{u}\circ f_u:\mc E\to\mc E_{\G_Q(u)}\,,\quad u\in Q\,.
\]
Since ${\phi}_u$ is a \Cex-conjugacy from $\F_Q^\diam(u)$ to $(\G_Q(u),\,{\phi}^\sharp_u\circ f_u)$ and $\F^\diam_Q(u)$ belongs to $\mr{SL}(\F^\diam)$, this  pre-marking satisfies
\begin{equation*}
(\G_Q(u),\,{\phi}^\sharp_u\circ f_u)\in \mr{SL}(\F^\diam)
\end{equation*}
 for each $u\in Q$.

As in \S\ref{marquages}, we denote
 by ${\mr{Mcg}}(\mc E^\diam)$ the {mapping class group} of $\mc E^\diam=(\mc E,\Sigma,\cdot)$, that is the group of  isotopy classes {$\dot\varphi$}  of  homeomorphisms  ${\varphi}:\mc E\iso \mc E$ leaving invariant the symmetric map $\cdot$ and  the set $\Sigma$. For each ${\dot\varphi\in \mr{Mcg}}(\mc E^\diam)$ let us consider the set 
\begin{equation*}
K_{\dot\varphi}:=\left\{u\in Q\;;\; g_u^{-1}\circ\phi^\sharp_u\circ f_u\in\dot\varphi\right\}.
\end{equation*}
Since $\mr{Mcg}(\mc E^\diam)$ is countable and 
\[
\bigcup_{\dot\varphi\in \mr{Mcg}(\mc E^\diam)} K_{\dot\varphi}= Q\,,
\]
there exists an element ${\dot\varphi_0\in \mr{Mcg}}(\mc E^\diam)$ such that $K_{{\dot\varphi_0}}$ is not contained in any countable union  of proper closed analytic subsets of $Q$. Let us consider the marked by $\mc E^\diam$ equisingular  global family 
\[
\un\G_Q^\diam:=\left(\un\G_Q,\, (G_u)_{u\in Q}\right)\,,\quad 
G_u:= g_u\circ{\varphi_0}\,.
\]
We highlight that 
\begin{equation}\label{conjKPhi}
{\phi}_u(\F_Q^\diam(u))=\G_Q^\diam(u) \,,\quad\hbox{ if }\quad u\in K_{{\dot\varphi_0}}\,,
\end{equation}
as in this case $g_u\circ{\varphi_0}$ is isotopic to ${\phi}^\sharp_u\circ f_u$. 
Therefore $\G_Q^\diam(u)$ belongs to $\mr{SL}(\F^\diam)$ when $u\in K_{{\dot\varphi_0}}$.
It follows from Remark \ref{extsl} that  $\un\G_Q^\diam$ belongs to $\mr{SL}_Q(\F^\diam)$ and we can consider the map $Q\ni u\mapsto [\G^\diam_Q(u)]\in\mr{Mod}([\F^\diam])$.

Let us now consider the map $\beta:\mr{Mod}([\F^\diam])\to\mbf D$ in the exact sequence (\ref{excsequmod1}).
By Corollary  \ref{BetaLambdacCt} there is
 $d\in\mbf D$ such that $\beta([\F^\diam_Q(u)])=d$, for every $u\in Q$. From (\ref{conjKPhi}) and Corollary \ref{BetaLambdacCt} we also have $\beta([\G^\diam_Q(u)])=d$, for every $u\in Q$.
Let us fix $u_1\in K_{{\dot\varphi_0}}$ and $z_1\in\C^\tau$ satisfying
\[
[\F_Q^\diam(u_1)]=[\G_Q^\diam(u_1)]=\Lambda(z_1)\cdot\zeta(d) 
\]
By Theorem  \ref{inivfaible} there exist two holomorphic maps 
\[
\lambda : Q\longrightarrow \C^\tau\,,\qquad
\lambda' : Q\longrightarrow \C^\tau\,,
\]
satisfying $\lambda(u_1)=\lambda'(u_1)=z_1$ and
\begin{equation}\label{LlambdaFG}
\Lambda(\lambda(u))\cdot\zeta(d)=[\F_Q^\diam(u)]\,,\quad
\Lambda(\lambda'(u))\cdot\zeta(d)=[\G_Q^\diam(u)]\,,
\quad
u\in Q\,,
\end{equation}
where $\Lambda:\C^\tau\to\mr{Mod}([\F^\diam])$ is the map in (\ref{excsequmod1}).
From (\ref{conjKPhi})  for $u\in K_{{\dot\varphi_0}}$ we have
\[
\Lambda(\lambda(u))\cdot\zeta(d)=\Lambda(\lambda'(u))\cdot\zeta(d)
\]
and  $\lambda(u)-\lambda'(u)$ belongs to $\ker(\Lambda)$.
For each $u\in K_{\dot\varphi_0}$ we fix $N\in\Z^p$ such that $\lambda(u)=\lambda'(u)+\alpha(N)$. We have
\[K_{\dot\varphi_0}\subset\bigcup_{N\in\Z^p}L_N\quad\text{
where}\quad
L_N=\{
u\in Q\;;\;\lambda(u)-\lambda'(u)=\alpha(N)
\}.
\]
Since each $L_N$ is a closed analytic subset of $Q$ and $K_{{\dot\varphi_0}}$ is not contained in a countable union of proper such sets, there exists $\wt N\in \Z^p$ such that $L_{\wt N}=Q$. 
Consequently
\[
\lambda(u)=\alpha(\wt N)+\lambda'(u)\quad \hbox{for every}\quad u\in Q\,.
\]
Then equalities (\ref{LlambdaFG}) give the required equalities (\ref{goodmarking}); that ends the proof.
\end{dem2}

\subsection{Conjugacies of families versus conjugacies  of deformations} According to Remark \ref{prglobalmark}, any deformation of a marked foliation may be canonically endowed with a marking.
We will see that under finite type assumptions, the notion of conjugacy of deformations is equivalent to  that of  conjugacy of their marked associated families.

\begin{teo}\label{equideffam} 
 Let us consider a finite type foliation $\F$ which is a generalized curve,  
 $f:\mc E\to \mc E_\F$ a marking of $\F$, $(\un\F_{Q^\point},\iota)$ and $(\un\G_{Q^\point},\delta)$ two equisingular deformations of $\F$ over a germ of manifold $Q^\point:=(Q,\tu)$. 
Let us denote by  $\un\F_{Q^\point}^\diam$ and $\un\G_{Q^\point}^\diam$ the families $\un\F_{Q^\point}$ resp. $\un\G_{Q^\point}$ endowed with the markings induced by the markings $\iota^\sharp\circ f$ and $\delta^\sharp\circ f $  of their special fibers.
Then the following properties are equivalent:
\begin{enumerate}
\item\label{conjcomp} there is a \Cex-conjugacy $\phi$ between the germs of families of $\un\F_{Q^\point}$ and $\un\G_{Q^\point}$ such that the lifting of $\delta^{-1}\circ\phi\circ\iota$ through the reduction of singularities of $\F$ leaves invariant each irreducible component of $\mc E_\F$,
\item\label{conjassfam} the marked families $\un\F_{Q^\point}^\diam$ and $\un\G_{Q^\point}^\diam$ are \Cex-conjugated,
\item\label{conjdeformations} the deformations $(\un\F_{Q^\point},\iota)$ and $(\un\G_{Q^\point},\delta)$ are \Cex-conjugated.
\end{enumerate} 
\end{teo}

The proof of this theorem is based on the following property of the pull-back map~$\phi^*$ introduced in Theorem~\ref{phi*}. We recall, see Definition~\ref{conj-def}, that
\[\mr{Def}_\F^{P^\point}=\{\text{equisingular deformations of $\F$ over $P^\point$}\}/\sim_{\mc C^{\mr{ex}}}.\]

\begin{lema}\label{sigmaid}
Let $\F$ be a germ of foliation  and let $\phi$ be a \Cex-homeomorphism that conjugates $\F$ to itself. Assume that $\F$ is a finite type generalized curve and that the lifting $\phi^\sharp$ of $\phi$ through the reduction of singularities of $\F$ leaves invariant each irreducible component of the exceptional divisor $\mc E_\F$.
Then for any pointed manifold $P^\point$ the pull-back map $\phi^\ast:\mr{Def}^{P^\point}_{\F}\iso \mr{Def}^{P^\point}_{\F}$ is the identity map.
\end{lema}

We will proceed now to prove Theorem~\ref{equideffam} using Lemma~\ref{sigmaid} which will be proven at the end of the section.  

\begin{dem2}{of Theorem~\ref{equideffam} }
The implications $(\ref{conjdeformations})\Rightarrow(\ref{conjassfam})\Rightarrow(\ref{conjcomp})$ are trivial. 
To see the implication $(\ref{conjcomp})\Rightarrow(\ref{conjdeformations})$, 
 let us denote by $\phi_{\tu}$ the restriction of $\phi$ to the fiber over $\tu$ and let us consider the
\Cex-homeomorphism $\psi:=\iota\circ\delta^{-1}\circ\phi_{\tu}$ which is an automorphism of $\F_{Q^\point}(\tu)$. Notice that $\psi$  is conjugated by $\iota$ to $ \delta^{-1}\circ\phi\circ\iota$, consequently its lifting $\psi^\sharp$ through the reduction of singularities of $\F_{Q^\point}(\tu)$ leaves invariant each irreducible component of the exceptional divisor  $\mc E_{\F_{Q^\point}(\tu)}$.

We will see now that the automorphism $\psi$ extends to a \Cex-automorphism $\Psi$ of the family $\un\F_{Q^\point}$.
 Let $j$ be the canonical embedding of the special fiber $\F_{Q^\point}(\tilde u)$ in the family $\un\F_{Q^\point}$, so that $(\un\F_{Q^\point},j)$ is a deformation of $\F_{Q^\point}(\tilde u)$.
 Theorem \ref{phi*} provides a deformation $(\un{\mc {K}}_{Q^\point},k)$  of $\F_{Q^\point}(\tilde u)$ that is conjugated to $(\un\F_{Q^\point},j)$ by a \Cex-homeomorphism germ $\Theta:\un{\mc {K}}_{Q^\point}\to\un\F_{Q^\point}$ with $\Theta\circ k=j\circ\psi$. By definition of $\psi^*$ we have $[\un{\mc {K}}_{Q^\point},k]= \psi^*([\un\F_{Q^\point},j])$.
  According to  Lemma~\ref{sigmaid} we have
$\psi^\ast([\un\F_{Q^\point},j])=[\un\F_{Q^\point},j]$. 
This means that $(\un{\mc {K}}_{Q^\point},k)$ is $\mc C^\mr{ex}$-conjugated to $(\un\F_{Q^\point},j)$  and  there is a \Cex-homeomorphism germ $\Xi$ such that 
$\Xi(\un\F_{Q^\point})=\un{\mc {K}}_{Q^\point}$ and $\Xi\circ j=k$. 
Hence $\Theta\circ\Xi\circ j=j\circ\psi$, i.e. $\Psi:=\Theta\circ\Xi$  is a \Cex-automorphism of $\un\F_{Q^\point}$ which extends~$\psi$.

To end the proof we notice that the \Cex-homeomorphism $\Phi:=\phi\circ\Psi^{-1}$ satisfies $\Phi(\un\F_{Q^\point})=\un\G_{Q^\point}$ and 
\[\Phi\circ\iota=\Phi_{\tu}\circ\iota=\phi_{\tu}\circ\Psi_{\tu}^{-1}\circ\iota=\phi_{\tu}\circ\psi^{-1}\circ\iota=\phi_\tu\circ\phi_\tu^{-1}\circ\delta\circ\iota^{-1}\circ\iota=\delta.\]
Hence the $\Phi$ is a \Cex-conjugacy between the deformations $(\un\F_{Q^\point},\iota)$ and $(\un\G_{Q^\point},\delta)$.
\end{dem2}

Before starting the proof of Lemma~\ref{sigmaid}, let us recall the functor representation result in \cite{MMSII} that we will use.\\

Let $\mbf{Fol}$ be the category of germs of  generalized curves on $(\C^2,0)$ and excellent conjugacies. We denote by $\mbf{Fol}_{\mbf{ft}}\subset\mbf{Fol}$ the full subcategory consisting in finite type foliations.
Let $\mbf{Man}^\point\subset\mbf{Set}^\point$ be the categories of pointed  complex manifolds and sets.
In  \cite[\S5.3 and \S1.2]{MMSII} we have introduced the contravariant functor
$\mbf{Fol_{ft}}\to\mbf{Man}^\point$, $\F\mapsto H^1(\A_{\F},\T_{\F})$. 
Any excellent conjugacy $\phi:\G\to\F$ induces the graph morphism $\A_\phi:\A_\G\to\A_\F$, $\star\mapsto \phi^\sharp(\star)$, which allows to define a morphism $[\phi^*]:H^1(\A_{\F},\T_{\F})\to H^1(\A_{\G},\T_{\G})$ in the following way:
\begin{equation*}\label{[phi*]}
[\phi^*]([X_{D,\langle D,D'\rangle}])=[Y_{D,\langle D,D'\rangle}],\quad
Y_{D,\langle D,D'\rangle}=(\phi^{\sharp})^*X_{\A_\phi(D,\langle D,D'\rangle)}\,.
\end{equation*}
We also considered the contravariant functor  $\mr{Fac}:\mbf{Man}^\point\times\mbf{Fol_{ft}}\to\mbf{Set}^\point$ defined by
\[\mr{Fac}:(P^\point,\F)\mapsto\mc O(P^\point,H^1(\A_{\F},\T_{\F})^\point)\,,\]
and if $\mu:Q^\point\to P^\point $ and $\phi:\G\to\F$ are morphisms in $\mbf{Man}^\point$ and $\mbf{Fol_{ft}}$ respectively,  then
 $\mr{Fac}_{\phi}^{\mu}$
 sends $\lambda\in \mc O(P^\point,H^1(\A_{\F},\T_{\F})^\point)$ to $[\phi^*]\circ\lambda\circ\mu\in\mc O(Q^\point,H^1(\A_{\G},\T_{\G})^\point)$.\\
 
On the other hand,  Theorem~\ref{phi*} 
allows us to consider the correspondence
\[\mr{Def}:(P^\point,\F)\mapsto\mr{Def}_\F^{P^\point}\]
together with the morphisms 
\[\mr{Def}_\phi^\mu:=\phi^*\circ\mu^*=\mu^*\circ\phi^*.\]
According to \cite[Theorem~3.11]{MMSII} 
$\mr{Def}:\mbf{Man}^\point\times\mbf{Fol_{ft}}\to \mbf{Set}^\point$ is a
contravariant functor. The main  result in \cite{MMSII} is:

\begin{teo}[{\cite[Theorem 6.3]{MMSII}}]\label{xi}
There is an isomorphism of functors 
\[\xi:\mr{Def}\iso\mr{Fac}.\]
\end{teo}

\bigskip

\begin{dem2}{of Lemma~\ref{sigmaid}} 
Thanks to Theorem~\ref{xi} it suffices to  prove that the  morphisms 
\[\mr{Fac}^{\mr{id}_{P^\point}}_{\phi} : 
\mc{O}(P^\point, H^1(\A_\F, \mc T_{\F})^\point)
\to
 \mc{O}(P^\point, H^1(\A_\F, \mc T_{\F})^\point)\,,
\quad
\lambda\mapsto [\phi^\ast]\circ \lambda\,,
\]
 are the identity maps.
 Indeed the naturality of $\xi$ gives the following commutative diagrams 
\[
\xymatrix{
\mr{Def}^{P^\point}_{\F}\ar[r]^{\phi^\ast}\ar[d]^{\rotatebox{90}{\ensuremath{\sim}}}_{\xi^{P^\point}_{\F}}& \mr{Def}^{P^\point}_{\F}\ar[d]_{\rotatebox{90}{\ensuremath{\sim}}}^{\xi^{P^\point}_{\F}}\\
\mc{O}(P^\point, H^1(\A_\F, \mc T_{\F})^\point)
\ar[r]_{\mr{Fac}^{\mr{id}_{P^\point}}_{\phi}}&
 \mc{O}(P^\point, H^1(\A_\F, \mc T_{\F})^\point)\,.
}
\]
The map $\mr{Fac}_\phi^{\mr{id}_{P^\point}}$ is the identity if and only if
the pull-back map $[\phi^\ast]:H^1(\A_\F, \mc T_{\F})\longrightarrow H^1(\A_\F, \mc T_{\F})$ is the identity. Since each irreducible component $D\in\Ve_{\A_\F}$ is fixed by $\phi^\sharp$,
the induced graph morphism $\A_{\phi}:\A_\F\to\A_\F$  is the identity map and 
the map
 \[[\phi^\ast]:H^1(\A_\F, \mc T_{\F})\longrightarrow H^1(\A_\F, \mc T_{\F})\]
 sends
 $ 
[(X_{D, \msf e})_{(D,\ge)\in\mc I_{\A_\F}}]$ into $
\big[(\phi^{\sharp\ast}(X_{D,\msf e}))_{(D,\ge)\in\mc I_{\A_\F}}\big]$.
Thanks to \cite[Remark~5.10]{MMSII}, it suffices to see that $\phi^{\sharp*}:\T_{\F}(\ge)\to\T_{\F}(\ge)$ is the identity for each $\ge=\langle D,D'\rangle\in\Ed_{\A_\F}$ such that $\T_\F(\ge)$ is one-dimensional.
The germ at $\{m\}=D\cap D'$ of the foliation $\F^\sharp$  is either linearizable non-resonant, or resonant normalizable and non-linearizable.  
Let us fix  a local chart  
\[(z_1,z_2): \Omega\iso \overline{\mb D}_r\times \ov{\mb D}_r,\qquad\overline{\mb D}_r=\{z\in\C: |z|\le r\},\] 
centered at the point  $m$, satisfying
 \[r>1,\quad z(m)=(0,0),\quad D\cup D'=\{z_1z_2=0\}\]
 such that the foliation   $\F^\sharp$ is defined on $\Omega$ either by the $1$-form $\omega=\omega_L$ or by $\omega=\omega_N$, with
\begin{itemize}
\item $\omega_{L} = az_2dz_1+b z_1dz_2$, with $
a,\,b\in \C^*$, $b/a\notin \mb Q\,$,
\item $\omega_N=
az_2(1+\zeta (z_1^az_2^b)^k)dz_1 + bz_1(1+(\zeta-1) (z_1^az_2^b)^k)dz_2$, $a,b,k\in \N^\ast
$,\ $\zeta\in\C$.
\end{itemize}
According to \cite[Lemma~5.4]{MMSII}  in both cases there exists an explicit holomorphic vector field $Z$ on $\Omega$, that is tangent to $\{z_1=1\}$ such that $[Z]$ generates $\T_{\F}(\ge)=\un{\mc B}_{\F,m}/\un{\mc X}_{\F,m}$:
\[Z=z_2\frac{\partial}{\partial z_2} \text{ if } \omega=\omega_L,\qquad Z=\frac{(z_1^az_2^b)^k}{1+\zeta(z_1^az_2^b)^k}z_2\frac{\partial}{\partial z_2} \text{ if }\omega=\omega_N\,.\]

Let us fix a point  $p\in \Omega$ with coordinates $z_1=\varepsilon$, $z_2=0$, where $\varepsilon\in \R_{>0}$ is sufficiently small so that $\phi^\sharp$ is holomorphic on $\Omega_{\varepsilon}:=\{|z_1|, |z_2|\leq \varepsilon\}\subset\Omega$ and $\phi^\sharp(\Omega_\varepsilon)\subset \Omega$. For $q\in\Omega$ we will denote by $Z_q$ the germ of $Z$ at  $q$. 
We must prove the equality $[\phi^{\sharp\,\ast} Z_m]=[Z_m]$ in $\T_\F(\ge)$, or equivalently that $Z_m - (D\phi^\sharp \cdot Z_m )\circ\phi^{\sharp\,-1}$ is tangent to $\F^\sharp$. 
We will use the following fact about the quotient sheaf $\un{\mc B}_{\F}/\un{\mc X}_\F$ of basic and tangent vector fields of $\F^\sharp$, cf. \S\ref{Suniv}:
\begin{enumerate}[\hspace{5mm}-]
\item {\it Let $X$ be a section of the sheaf $\un{\mc B}_{\F}/\un{\mc X}_\F$ restricted to a connected open subset $V$ of an invariant irreducible component  of $\mc E_\F$.  
If the germ of  $X$  at some point $p$ of $V$ is zero, then $X =0$. }
\end{enumerate}
Indeed if $p$ is a regular point, by local triviality, the section is zero along the whole regular part of $V$. The vanishing at the remaining singularities follows by analytic continuation. If $p$ is a singular point, then the germ of $X$ at a regular point close to $p$ is zero and we conclude similarly.

Thanks to this property it suffices to show that at the point $\phi^\sharp(p)$
 the vector field germ $Z_{\phi^\sharp(p)} - (D\phi^\sharp \cdot Z_p)\circ\phi^{\sharp\,-1}$ is tangent to $\F^\sharp$. Let us choose a simple path 
 \[
\gamma :[0,1]\to \{z_1\neq 0\,,\, z_2=0\}\,,\quad
\gamma(0)=p\,,\quad
\gamma(1)=\phi^\sharp(p)\,,
 \]
and a germ at $p$  of   holomorphic submersion 
$
I_p : (\Omega, p)\to (\C,0)$ constant on the leaves of $\F^\sharp$ whose restriction to $\{z_1=\varepsilon\}$ is equal to $z_2\varepsilon^{\frac{a}{b}}$. Let us denote by $I_{\phi^\sharp(p)} : (\Omega, \phi^\sharp(p))\to (\C,0)$ the analytic extension of $I_p$ along $\gamma$, which coincides with  its extension as first integral of $\F^\sharp$. The vector field $Z$ being basic, the  germ of  holomorphic  vector field on $(\C,0)$ 
\[
Z^\flat=z\DD{z}\quad \hbox{if}\quad \omega=\omega_L\,,\qquad \hbox{or} \qquad Z^\flat=\frac{z^{bk}}{1+\zeta z^{bk}}z\DD{z}\quad\hbox{if}\quad \omega=\omega_N\,,
\]
satisfies the relations
\begin{equation}\label{DIZ}
DI_p\cdot Z_{p} = Z^\flat\circ I_p\quad \hbox{and}\quad DI_{\phi^\sharp(p)}\cdot Z_{\phi^\sharp(p)}=Z^\flat\circ I_{\phi^\sharp(p)}\,,
\end{equation}
the second equality resulting from the first one by analytic extension. 
On the other hand, the germ of $\phi^\sharp$ at $p$ factorizes through the first integrals, inducing a biholomophism germ $\phi^\flat:(\C,0)\to(\C,0)$ such that 
\begin{equation}\label{defvarphiflat}
I_{\phi^\sharp(p)}\circ \phi^\sharp=\phi^\flat\circ I_p\,.
\end{equation}  
Using the chain rule we have:
 \begin{align}
 DI_{\phi^\sharp(p)}\cdot &\big((D\phi^\sharp\cdot Z_p)\circ\phi^{\sharp\,-1}\big)
 = \big((DI_{\phi^\sharp(p)}\circ\phi^\sharp)\cdot D\phi^\sharp\cdot Z_p\big)\circ\phi^{\sharp\,-1}\nonumber\\
&=\big( D(I_{\phi^\sharp(p)}\circ \phi^\sharp)\cdot Z_p\big)\circ\phi^{\sharp\,-1}
\stackrel{(\ref{defvarphiflat})}{=}\big( D( \phi^\flat\circ I_p)\cdot Z_p\big)\circ\phi^{\sharp\,-1}\nonumber\\
&
=\big((D\phi^\flat\circ I_p)\cdot D I_p\cdot Z_p\big)\circ\phi^{\sharp\,-1}\stackrel{(\ref{DIZ})}{=}
\big((D\phi^\flat\circ I_p)\cdot (Z^\flat\circ I_p)\big)\circ\phi^{\sharp\,-1}\nonumber\\
&
=(D\phi^\flat\cdot Z^\flat)\circ I_p\circ\phi^{\sharp\,-1}\,.\label{DIZ2}
 \end{align}
 Since $\phi^\sharp$ is defined on a neighborhood of the singular point $m$, $\phi^\flat$ commutes with the biholomorphism  
 of holonomy of $\F^\sharp$ along $D$ and around $m$. According to \cite[Proposition~6.10]{MMS} (cases (L) and (R)) there is some $t_1\in \C$ and a linear periodic map $\ell:\C\to\C$ such that:
 \begin{equation*}
  \phi^\flat=\ell\circ \exp(Z^\flat)[t_1]\,,\quad
 \ell^\ast (Z^\flat)=Z^\flat\,;
 \end{equation*}
therefore
$\phi^\flat$ leaves $Z^\flat$ invariant. Hence the equality  $D\phi^\flat\cdot Z^\flat=Z^\flat\circ \phi^\flat$ holds and using it we obtain
 \begin{align*}
  DI_{\phi^\sharp(p)}\cdot \big((D\phi^\sharp\cdot Z_p)\circ\phi^{\sharp\,-1}\big)
 &\stackrel{(\ref{DIZ2})}{=}
 (D\phi^\flat\cdot Z^\flat)\circ I_p\circ\phi^{\sharp\, -1}=
 Z^\flat\circ \phi^\flat\circ I_p\circ\phi^{\sharp\,-1}\\ &
 \stackrel{(\ref{defvarphiflat})}{=} Z^\flat\circ  I_{\phi^\sharp(p)}\circ \phi^\sharp  \circ\phi^{\sharp\,-1}
 \stackrel{(\ref{DIZ})}{=} DI_{\phi^\sharp(p)}\cdot Z_{\phi^\sharp(p)}\,.
\end{align*}
We finally have the equality
\[
  DI_{\phi^\sharp(p)}\cdot \big(Z_{\phi^\sharp(p)} - (D\phi^\sharp \cdot Z_p)\circ\phi^{\sharp\,-1}\big)=0\,,
\]
that shows that the vector field germ $Z_{\phi^\sharp(p)} - (D\phi^\sharp \cdot Z_p)\circ\phi^{\sharp\,-1}$  is tangent to $\F^\sharp$. 
\end{dem2}

\bibliographystyle{plain}

\end{document}